\documentclass[12pt, reqno]{amsart}

\usepackage[a4paper, centering, total={170mm,240mm}]{geometry}

\usepackage{amssymb,latexsym}
\usepackage{blindtext}
\usepackage[backgroundcolor=olive, linecolor=olive, textsize=tiny]{todonotes}
\usepackage{esint,dsfont}

\usepackage{palatino}

\usepackage[colorlinks=true, linkcolor=blue, citecolor=purple, urlcolor=blue]{hyperref}

\def\R{{\mathbb R}}
\def\N{\mathbb{N}}
\def\C{\mathbb{C}}
\def\Z{\mathbb{Z}}
\def\D{\mathbb{D}}
\def\T{\mathbb{T}}
\def\F{\mathbb{F}}

\newtheorem{prop}{\bf Proposition}[section]
\newtheorem{thm}[prop]{\bf Theorem}
\newtheorem{cor}[prop]{\bf Corollary}

\newtheorem{rmk}[prop]{\it Remark}

\newtheorem{ques}[prop]{\bf Question}
\newtheorem{conj}[prop]{\bf Conjecture}

\begin{document}

\title[An invitation to weak amenability]{{\bf\Large An invitation to weak amenability, after Cowling and Haagerup}}

%%    Information for first author
\author[I. Vergara]{Ignacio Vergara}
\address{Departamento de Matem\'atica y Ciencia de la Computaci\'on, Universidad de Santiago de Chile, Las Sophoras 173, Estaci\'on Central 9170020, Chile}

\email{ign.vergara.s@gmail.com}
\thanks{This work is supported by the FONDECYT project 3230024.}

%    General info
\makeatletter
\@namedef{subjclassname@2020}{%
  \textup{2020} Mathematics Subject Classification}
\makeatother

%
%\date{\today}


\begin{abstract}
We present an introduction to weak amenability for locally compact groups, and a survey of some of the most important results regarding this property.
\end{abstract}

%\maketitle

\begingroup
\def\uppercasenonmath#1{} % this disables uppercasing title
\let\MakeUppercase\relax % this disables uppercasing authors
\maketitle
\endgroup

{\hypersetup{linkcolor=black}
\tableofcontents
}

\section{{\bf Introduction}}

Weak amenability is an analytic property of groups that generalises amenability from the point of view of harmonic analysis and operator algebras. It was first introduced (under that name) by Cowling and Haagerup in \cite{CowHaa}. In that same article, they introduced what is now called the Cowling--Haagerup constant of a group, which serves as a quantification of how strongly this property is satisfied. More importantly, this constant provides an invariant at the von Neumann algebra and $\mathbf{C}^*$-algebra levels, making this property particularly interesting for the study of operator algebras.

A (discrete) group $G$ is said to be amenable if there is a left-invariant mean on the space $\ell^\infty(G)$. This property can be characterised by the fact that the constant function $1$ on $G$ can be approximated by finitely supported, positive definite functions. Weak amenability is defined by relaxing this approximation, and requiring the functions to be uniformly bounded for a certain multiplier norm instead. Operator algebras and harmonic analysis are at the heart of this definition, but as we will see, there are deep connections with geometric group theory, the structure of Lie groups, and measured group theory.

The aim of these notes is twofold. Firstly, they are intended as an introduction to weak amenability and the completely bounded approximation property, which is its operator algebraic counterpart (Sections \ref{Sec_OA_CBAP}--\ref{Sec_wa}). The second objective is to provide an account on some of the most important results around this property (Sections \ref{Sec_Lie}--\ref{Sec_AP_HK}).

We begin by discussing the completely bounded approximation property (CBAP) in Section \ref{Sec_OA_CBAP}, as well as the operator algebraic language necessary to understand it.

Section \ref{Sec_Fourier_alg} is devoted to harmonic analysis on locally compact groups. More concretely, we introduce the Fourier algebra and the notion of Fourier multiplier.

Weak amenability is finally defined in full detail in Section \ref{Sec_wa}. Some stability properties are also discussed, as well as its connection to the CBAP.

Section \ref{Sec_Lie} is devoted to Lie groups, which are at the genesis of weak amenability, and for which this property is very well understood nowadays. A series of very deep results are condensed in three theorems stated in that section, together with the history of how they were obtained.

Very early on, the geometry of trees proved to be useful to the study of weak amenability. This opened the door to the use of more tools from geometric group theory, including hyperbolicity and the structure of $\operatorname{CAT}(0)$ cube complexes. These topics are discussed in Section \ref{Sec_ggt}.

An even stronger connection exists with the realm of measured group theory, as witnessed by the fact that the Cowling--Haagerup constant is a measure equivalence invariant. Section \ref{Sec_mgt} is devoted to this subject.

In Section \ref{Sec_obstr}, we focus groups that are not weakly amenable, and the known obstructions to satisfying this property.

In Section \ref{Sec_Haag_prop}, we discuss the Haagerup property and an open problem relating it to weak amenability, often referred to as Cowling's conjecture.

Finally, in Section \ref{Sec_AP_HK}, we introduce an even weaker version of weak amenability, known as the approximation property (AP) of Haagerup and Kraus. Some examples are briefly discussed, as well as the problem of convoluters and pseudo-measures, which has been solved for groups satisfying AP.

Most of the results are presented without proof, but all the references are precisely indicated for the reader who wishes to go further. The only exception is Theorem \ref{Thm_wa_CBAP}, which connects weak amenability with the CBAP. We include its proof because it is a central piece of this theory, and because it emphasises the connections between all the topics discussed in Sections \ref{Sec_OA_CBAP}--\ref{Sec_wa}.

We must point out that these notes do not provide an exhaustive account on weak amenability, and there are certainly some important topics that are not covered. In particular, the connections with the deformation/rigidity theory in von Neumann algebras, and weak amenability for quantum groups and noncommutative dynamical systems do not appear in this text.

For the reader, some basic knowledge of functional analysis and locally compact groups will be assumed; our preferred references for this are the books \cite{Rud} and \cite{Fol} respectively. Although knowing the basics of operator algebras may certainly be helpful, we made an effort for this not to be a prerequisite, and we included all the relevant definitions and results, opting for a simple and (hopefully) pedagogical presentation. As a result, the experts will probably find those sections sorely lacking, but we believe this to be for the greater good.

\section{{\bf Operator algebras and the CBAP}}\label{Sec_OA_CBAP}

Weak amenability is a property of groups, but the main motivation behind its definition is the study of operator algebras. In this section, we introduce the completely bounded approximation property (CBAP) for $\mathbf{C}^*$-algebras and its analogue (w*-CBAP) for von Neumann algebras. As we will see in Section \ref{Sec_wa}, these properties provide equivalent formulations of weak amenability for discrete groups.

\subsection{Operator algebras and completely bounded maps}
We quickly review here the definitions of $\mathbf{C}^*$-algebras, von Neumann algebras and completely bounded maps; we refer the reader to \cite{Mur}, \cite{BroOza} and \cite{Pis} for more detailed expositions of these topics.

Let $\mathcal{H}$ be a (complex) Hilbert space and let $\mathbf{B}(\mathcal{H})$ denote the algebra of bounded linear operators on $\mathcal{H}$, endowed with the operator norm:
\begin{align*}
\|T\|=\sup\{\|T\xi\|\ \mid\ \xi\in\mathcal{H},\ \|\xi\|\leq 1\},\quad \forall T\in\mathbf{B}(\mathcal{H}).
\end{align*}
For every $T\in\mathbf{B}(\mathcal{H})$, the adjoint operator $T^*$ is the unique element of $\mathbf{B}(\mathcal{H})$ satisfying
\begin{align*}
\langle T^*\xi,\eta\rangle=\langle\xi,T\eta\rangle,\quad\forall\xi,\eta\in\mathcal{H}.
\end{align*}
A $\mathbf{C}^*$-algebra is a norm-closed subalgebra of $\mathbf{B}(\mathcal{H})$ that is closed under taking adjoints. A von Neumann algebra is a $\mathbf{C}^*$-algebra $\mathcal{A}\subseteq\mathbf{B}(\mathcal{H})$ such that $\mathrm{Id}_{\mathcal{H}}\in\mathcal{A}$ and which is closed for the strong operator topology (SOT), i.e. the topology generated by the family of semi-norms
\begin{align*}
T\in\mathbf{B}(\mathcal{H})\mapsto\|T\xi\|,\quad (\xi\in\mathcal{H}).
\end{align*}

Let $\mathcal{A}\subseteq\mathbf{B}(\mathcal{H})$ be a $\mathbf{C}^*$-algebra. For every $n\geq 1$, let $M_n(\mathcal{A})$ be the algebra of $\mathcal{A}$-valued $n\times n$ matrices, endowed we the usual matrix multiplication. There is a natural $\mathbf{C}^*$-algebra structure on $M_n(\mathcal{A})$ given by the inclusion
\begin{align*}
M_n(\mathcal{A})\subseteq M_n(\mathbf{B}(\mathcal{H}))=\mathbf{B}(\mathcal{H}^n),
\end{align*}
where $\mathcal{H}^n=\mathcal{H}\oplus\cdots\oplus\mathcal{H}$ is the direct sum of $n$ copies of $\mathcal{H}$. Let $\mathcal{B}$ be another $\mathbf{C}^*$-algebra, and let $u:\mathcal{A}\to\mathcal{B}$ be a linear map. We can define $u_n:M_n(\mathcal{A})\to M_n(\mathcal{B})$ by
\begin{align}\label{maps_u_n}
u_n((x_{ij})_{i,j})=(u(x_{ij}))_{i,j},\quad\forall (x_{ij})_{i,j}\in M_n(\mathcal{A}).
\end{align}
We say that $u$ is completely bounded if
\begin{align*}
\|u\|_{\mathrm{cb}}=\sup_n\|u_n\|<\infty.
\end{align*}
We call $\|u\|_{\mathrm{cb}}$ the completely bounded norm of $u$.

A very important class of completely bounded maps is given by the completely positive ones. An element $x\in\mathcal{A}$ is said to be positive, and we write $x\geq 0$, if $x=y^*y$ for some $y\in\mathcal{A}$. A map $u:\mathcal{A}\to\mathcal{B}$ is positive if it sends positive elements to positive elements. We say that $u$ is completely positive if the map \eqref{maps_u_n} is positive for every $n$. The following result, due essentially to Stinespring \cite{Sti}, connects these two classes of maps; see \cite[Theorem 1.22]{Pis} for a proof.

\begin{thm}
Let $\mathcal{A}$ and $\mathcal{B}$ be $\mathbf{C}^*$-algebras. Let $u:\mathcal{A}\to\mathcal{B}$ be a completely positive linear map. Then $u$ is completely bounded and
\begin{align*}
\|u\|_{\mathrm{cb}}=\|u\|=\sup\{\|u(x)\| \ \mid\ x\geq 0, \ \|x\|\leq 1\}.
\end{align*}
If, in addition, $\mathcal{A}$ is unital, then
\begin{align*}
\|u\|_{\mathrm{cb}}=\|u(1)\|.
\end{align*}
\end{thm}

\subsection{Operator algebras associated to groups}
Let $\mathcal{H}$ be a Hilbert space. The unitary group of $\mathcal{H}$ is
\begin{align*}
\mathbf{U}(\mathcal{H})=\{T\in\mathbf{B}(\mathcal{H})\ \mid\ T\ \text{ invertible and }\ T^*=T^{-1}\}.
\end{align*}
Now let $G$ be a locally compact group. A unitary representation of $G$ on $\mathcal{H}$ is a group homomorphism $\pi:G\to\mathbf{U}(\mathcal{H})$ that is continuous for the strong operator topology, meaning that the map
\begin{align*}
s\in G \longmapsto \pi(s)\xi \in\mathcal{H}
\end{align*}
is continuous for every $\xi\in\mathcal{H}$.

Let $G$ be a locally compact group endowed with a left Haar measure; we refer the reader to \cite[\S 2]{Fol} for details. The left regular representation $\lambda:G\to\mathbf{U}(L^2(G))$ is defined by
\begin{align}\label{left_reg_rep}
\lambda(s)f(t)=f(s^{-1}t),\quad\forall f\in L^2(G),\ \forall s,t\in G.
\end{align}
A change of variables shows that $\lambda(s)^*=\lambda(s^{-1})$ for every $s\in G$. Furthermore, $L^1(G)$ is a Banach algebra for the convolution product, and it also has a natural involution given by
\begin{align*}
g^*(s)=\overline{g(s^{-1})}\Delta(s^{-1}),\quad \forall g\in L^1(G),\ \forall s\in G,
\end{align*}
where $\Delta:G\to(0,\infty)$ is the modular function; see \cite[\S 2.4]{Fol} for details. This allows us to extend $\lambda$ to a $\ast$-homomorphism $\lambda:L^1(G)\to\mathbf{B}(L^2(G))$ in the following way. For all $g\in L^1(G)$, $f\in L^2(G)$, $t\in G$, we define
\begin{align*}
\lambda(g)f(t)=\int_G g(s)f(s^{-1}t)\, ds,
\end{align*}
where $ds$ denotes the integration with respect to the Haar measure. In other words,
\begin{align*}
\lambda(g)f=g\ast f,\quad\forall g\in L^1(G),\ \forall f\in L^2(G).
\end{align*}
It follows from the definition that
\begin{align*}
\|\lambda(g)\|_{\mathbf{B}(L^2(G))}\leq\|g\|_1,\quad \forall g\in L^1(G).
\end{align*}
The reduced group $\mathbf{C}^*$-algebra $C_\lambda^*(G)$ is defined as the norm-closure of $\lambda(L^1(G))$ in $\mathbf{B}(L^2(G))$. The group von Neumann algebra $LG$ is defined analogously by taking the SOT-closure instead.

\begin{rmk}
If $G$ is a discrete group, the Haar measure is nothing but the counting measure on $G$. In particular, the group algebra $\C[G]$ is dense in $\ell^1(G)$. Hence the algebras $C_\lambda^*(G)$ and $LG$ can be defined similarly as completions of $\lambda(\C[G])$.
\end{rmk}

\subsection{Historical motivation: approximation properties of Banach spaces}\label{Ssec_AP_Banach}
The origins of weak amenability can be traced back to Haagerup's seminal work \cite{Haa} on the reduced $\mathbf{C}^*$-algebra of the free group on 2 generators $C_\lambda^*(\F_2)$. The main result of \cite{Haa} says that $C_\lambda^*(\F_2)$ has the metric approximation property (MAP), which is a strengthening of the approximation property (AP), explored by Grothendieck in \cite{Gro}. Afterwards, de Canni\`ere and Haagerup \cite{deCHaa} showed that $C_\lambda^*(\F_2)$ satisfies an even stronger condition involving completely contractive maps. In what follows, we briefly discuss the aforementioned properties; for a much more detailed presentation, we refer the reader to \cite[\S 1.e]{LinTza}.

Let $E$ be a Banach space. We say that $E$ has the approximation property (AP) if, for every compact subset $K\subset E$ and every $\varepsilon>0$, there is a finite-rank operator $T\in\mathbf{B}(E)$ such that
\begin{align*}
\|Tx-x\|\leq\varepsilon,\quad\forall x\in K.
\end{align*}
Every Hilbert space (and, more generally, every space with a Schauder basis) has the AP, and for a long time the question of whether every Banach space has the AP remained open. The problem was finally settled by Enflo \cite{Enf} who provided a counterexample.

This property can be strengthened by requiring the approximations of the identity to be bounded. More precisely, we say that a Banach space $E$  has the bounded approximation property (BAP) if there is $C\geq 1$ such that, for every compact subset $K\subset E$ and every $\varepsilon>0$, there is a finite-rank operator $T\in\mathbf{B}(E)$ such that
\begin{align}\label{BAP_bound}
\|T\|\leq C,
\end{align}
and
\begin{align*}
\|Tx-x\|\leq\varepsilon,\quad\forall x\in K.
\end{align*}
The uniform bound \eqref{BAP_bound} allows one to see that BAP is equivalent to the existence of a net of finite rank operators $(T_i)_{i\in I}$ in $\mathbf{B}(E)$ such that
\begin{align*}
\|T_i\|\leq C, \quad\forall i\in I,
\end{align*}
and
\begin{align*}
\|T_ix-x\|\to 0,\quad\forall x\in E.
\end{align*}
In the particular case in which the condition above is satisfied with $C=1$, we say that $E$ has the metric approximation property (MAP). The implications
\begin{align*}
\text{MAP}\implies\text{BAP}\implies\text{AP}
\end{align*}
are strict; see \cite[\S 1.e]{LinTza} for details. We now state the main result of \cite{Haa}.

\begin{thm}[Haagerup]
Let $\F_2$ be the free group on 2 generators. The $\mathbf{C}^*$-algebra $C_\lambda^*(\F_2)$ has the MAP.
\end{thm}

\begin{rmk}
This result provided the first example of a non nuclear $\mathbf{C}^*$-algebra satisfying MAP. We say that a $\mathbf{C}^*$-algebra $\mathcal{A}$ is nuclear if, for every $\mathbf{C}^*$-algebra $\mathcal{B}$, the tensor product $\mathcal{A}\otimes\mathcal{B}$ admits a unique norm for which its completion is a $\mathbf{C}^*$-algebra. This property turns out to be equivalent to the CPAP discussed below; see \cite[\S 4.2]{Pis}.
\end{rmk}

\subsection{The completely bounded approximation property}
Now we turn to a notion of approximation that is better suited to the study of $\mathbf{C}^*$-algebras. We say that a $\mathbf{C}^*$-algebra $\mathcal{A}$ has the completely bounded approximation property (CBAP) if there is a constant $C\geq 1$ and a net of finite rank linear maps $T_i:\mathcal{A}\to\mathcal{A}$ ($i\in\ I$) such that
\begin{align*}
\|T_i\|_{\mathrm{cb}}\leq C, \quad\forall i\in I,
\end{align*}
and
\begin{align*}
\|T_ix-x\|\to 0,\quad\forall x\in \mathcal{A}.
\end{align*}
We define $\boldsymbol\Lambda_{\mathrm{cb}}(\mathcal{A})$ as the infimum of all $C\geq 1$ such that the condition above holds. If $\boldsymbol\Lambda_{\mathrm{cb}}(\mathcal{A})=1$, we say that $\mathcal{A}$ has the completely contractive approximation property (CCAP). For a $\mathbf{C}^*$-algebra without the CBAP, we simply define $\boldsymbol\Lambda_{\mathrm{cb}}(\mathcal{A})=\infty$.

A stronger property than the CCAP is the following. We say that a $\mathbf{C}^*$-algebra $\mathcal{A}$ has the completely positive approximation property (CPAP) if there is a net of finite rank, completely positive linear maps $T_i:\mathcal{A}\to\mathcal{A}$ such that
\begin{align*}
\|T_ix-x\|\to 0,\quad\forall x\in \mathcal{A}.
\end{align*}
We have the following implications
\begin{align*}
\text{CPAP}\implies\text{CCAP}\implies\text{CBAP}.
\end{align*}
As we will see, both these implications are strict.

\subsection{The weak*-completely bounded approximation property}
Although von Neumann algebras are $\mathbf{C}^*$-algebras too, it is more convenient to look at them as an entirely different class, in the same way as we view measure spaces and topological spaces as different objects. We refer the reader to the beginning of Chapter III of \cite{Bla} for a discussion about this analogy. What is important here is the fact that the norm topology is not so well suited for studying von Neumann algebras, so in order to define analogues of the properties above for von Neumann algebras, we need to restrict our attention to weak*-continuous maps.

Let $\mathcal{H}$ be a Hilbert space, and let $\mathbf{S}^1(\mathcal{H})$ denote the space of trace-class operators; see e.g. \cite[\S I.8.5]{Bla}. Then $\mathbf{B}(\mathcal{H})$ can be identified with the dual space of $\mathbf{S}^1(\mathcal{H})$ by the duality
\begin{align*}
\langle x,y\rangle = \operatorname{Tr}(xy),\quad\forall x\in \mathbf{B}(\mathcal{H}),\ \forall y\in \mathbf{S}^1(\mathcal{H}).
\end{align*}
Now let $\mathcal{M}\subseteq\mathbf{B}(\mathcal{H})$ be a von Neumann algebra. Then $\mathcal{M}$ is a $\sigma(\mathbf{B}(\mathcal{H}),\mathbf{S}^1(\mathcal{H}))$-closed subspace of $\mathbf{B}(\mathcal{H})$. It is therefore itself a dual space with predual
\begin{align*}
\mathcal{M}_*=\mathbf{S}^1(\mathcal{H})/({}^\perp\!\mathcal{M}),
\end{align*}
where
\begin{align*}
{}^\perp\!\mathcal{M}=\{y\in \mathbf{S}^1(\mathcal{H})\ \mid\ \operatorname{Tr}(xy)=0\ \ \forall x\in\mathcal{M}\}.
\end{align*}
This notation is justified by the fact that von Neumann algebras have a unique predual. The space $\mathcal{M}_*$ can also be identified with the space of normal states on $\mathcal{M}$; see \cite[\S III.2]{Bla} for a detailed discussions of all these facts. Thus it makes sense to speak about the weak*-topology of $\mathcal{M}$.

We say that a von Neumann algebra $\mathcal{M}$ has the weak*-completely bounded approximation property (w*-CBAP) if there is a constant $C\geq 1$ and a net of finite rank, weak*-weak*-continuous linear maps $T_i:\mathcal{M}\to\mathcal{M}$ ($i\in\ I$) such that
\begin{align*}
\|T_i\|_{\mathrm{cb}}\leq C, \quad\forall i\in I,
\end{align*}
and
\begin{align*}
\langle T_ix,y\rangle\to \langle x,y\rangle,\quad\forall x\in \mathcal{M},\ \forall y\in\mathcal{M}_*.
\end{align*}
As before, we define $\boldsymbol\Lambda_{\mathrm{w^*cb}}(\mathcal{M})$ as the infimum of all $C\geq 1$ such that the condition above holds.

\begin{rmk}
The notation $\boldsymbol\Lambda_{\mathrm{w^*cb}}$ is not standard, and usually the same symbol as the one for the CBAP is used in this context. However, since von Neumann algebras can also be viewed as $\mathbf{C}^*$-algebras, it seems pertinent to introduce new nomenclature in order to avoid ambiguities.
\end{rmk}

We define w*-CCAP and w*-CPAP analogously, and we get
\begin{align*}
\text{w*-CPAP}\implies\text{w*-CCAP}\implies\text{w*-CBAP}.
\end{align*}
Again, we will see that both implications are strict.

\section{{\bf The Fourier algebra and multipliers}}\label{Sec_Fourier_alg}

An important theorem of Leptin \cite{Lep} characterises amenability of a group in terms of approximations of the identity of its Fourier algebra. In this section, we introduce the Fourier algebra and its multipliers. These objects will be necessary for defining weak amenability as a generalisation of amenability.

\subsection{Fourier analysis on $\Z$}\label{Ssec_fa_Z}
In order to motivate the definition of the Fourier algebra for general locally compact groups, we begin by analysing the case of $\Z$. Let $\T$ be the compact group $\R/\Z$, which is isomorphic to
\begin{align*}
\mathbf{U}(1)=\{z\in\C\ \mid\ |z|=1\}.
\end{align*}
The Haar measure of $\T$ is nothing but the Lebesgue measure. For every $f\in L^1(\T)$, we can define its Fourier coefficients
\begin{align*}
\hat{f}(n)=\int_{\T}f(z)z^{-n}\, dz,\quad\forall n\in\Z.
\end{align*}
The assignment $f\mapsto\hat{f}$ is the Fourier transform, and it defines a Banach algebra homomorphism $\mathcal{F}:L^1(\T)\to c_0(\Z)$, where $L^1(\T)$ is endowed with the convolution product, and $c_0(\Z)$ is the space of functions on $\Z$ vanishing at infinity endowed with the pointwise product. The Fourier algebra of $\Z$ is defined as the image of $\mathcal{F}$:
\begin{align*}
A(\Z)=\mathcal{F}(L^1(\T))=\{\hat{f}\ \mid\ f\in L^1(\T)\}.
\end{align*}
All these ideas can be extended to any abelian discrete group; see \cite[\S 4]{Fol} for details. Hence we have a natural definition of the Fourier algebra for all such groups. However, we would like to have one definition that can be applied to any locally compact group. This will be possible thanks to the following observation. For every $f\in L^1(\T)$, we can find $g,h\in L^2(\T)$ such that
\begin{align*}
f=gh \qquad\text{and}\qquad \|f\|_1=\|g\|_2\|h\|_2.
\end{align*}
This allows us to write
\begin{align*}
\hat{f}=\mathcal{F}(gh)=\hat{g}\ast\hat{h}.
\end{align*}
Since $\mathcal{F}:L^2(\T)\to\ell^2(\Z)$ is an isomorphism, we get the following description:
\begin{align}\label{A(Z)}
A(\Z)=\{u\ast v\ \mid\ u,v\in\ell^2(\Z)\}.
\end{align}
Furthermore, defining
\begin{align}\label{norm_A(Z)}
\|a\|_{A(\Z)}=\inf\left\{\|u\|_2\|v\|_2\ \mid\ a=u\ast v,\ u,v\in\ell^2(\Z)\right\},\quad\forall a\in A(\Z),
\end{align}
we see that this is exactly the quotient norm induced from $\mathcal{F}:L^1(\T)\to A(\Z)$. From \eqref{A(Z)} and \eqref{norm_A(Z)}, we can forget about the duality $(\T,\Z)$, and work directly on $\Z$, which will allow us to define $A(G)$ in general.

We go back now to the point of view of Fourier analysis in order to define a larger algebra that will be relevant for our purposes. Recall that $L^1(\T)$ is a two-sided ideal in the algebra of complex Randon measures $M(\T)$; see \cite[\S 2.5]{Fol}. Moreover, the Fourier transform can be defined for $\mu\in M(\T)$ as follows:
\begin{align*}
\hat{\mu}(n)=\int_{\T}z^{-n}\,d\mu(z),\quad\forall n\in\Z.
\end{align*}
The Fourier--Stieltjes algebra of $\Z$ is
\begin{align*}
B(\Z)=\mathcal{F}(M(\T))=\{\hat{\mu}\ \mid\ \mu\in M(\T)\}.
\end{align*}
Recall that every $\mu\in M(\T)$ can be decomposed as a linear combination of four positive measures. Hence, by Bochner's theorem (see e.g. \cite[Theorem 4.19]{Fol}), $B(\Z)$ is the linear span of the positive definite functions on $\Z$, i.e. all the functions $\varphi\in\ell^\infty(\Z)$ such that the matrix $(\varphi(n_i-n_j))_{i,j=1}^m$ is positive definite for every finite sequence $n_1,\ldots,n_m\in\Z$. Again, this gives us a way of forgetting that we started from measures on $\T$, and work directly on $\Z$.

\subsection{Unitary representations and the Fourier--Stieltjes algebra}
The Fourier algebra and the Fourier--Stieltjes algebra of a general locally compact group were introduced by Eymard in \cite{Eym}; we refer the reader to \cite{KanLau} for another very good reference.

Let $G$ be a locally compact group, and let $\pi:G\to\mathbf{U}(\mathcal{H})$ be a unitary representation. We say that $\varphi:G\to\C$ is a coefficient of $\pi$ if there are $\xi,\eta\in\mathcal{H}$ such that
\begin{align}\label{coeff_unit}
\varphi(s)=\langle\pi(s)\xi,\eta\rangle,\quad\forall s\in G.
\end{align}
Let $B(G)$ denote the space of all coefficients of unitary representations of $G$. We call it the Fourier--Stieltjes algebra of $G$, and we endow it with the norm
\begin{align*}
\|\varphi\|_{B(G)}=\inf\|\xi\|\|\eta\|,
\end{align*} 
where the infimum is taken over all the decompositions as in \eqref{coeff_unit}. With this norm, $B(G)$ becomes a Banach algebra for pointwise operations; see e.g. \cite[Theorem 2.1.11]{KanLau}. Let $C_b(G)$ be the commutative Banach algebra of continuous bounded functions on $G$. We have a contractive inclusion $B(G)\subseteq C_b(G)$, i.e.
\begin{align*}
\|\varphi\|_\infty\leq\|\varphi\|_{B(G)},\quad\forall \varphi\in B(G).
\end{align*}
Now we relate this definition to the one in Section \ref{Ssec_fa_Z}. We say that $\varphi\in C_b(G)$ is positive definite if, for any finite sequences $s_1,\ldots,s_n\in G$, $z_1,\ldots,z_n\in\C$,
\begin{align*}
\sum_{i,j=1}^n \varphi(s_j^{-1}s_i)z_i\overline{z_j}\geq 0.
\end{align*}
This turns out to be equivalent to the fact that $\varphi$ can be written as
\begin{align*}
\varphi(s)=\langle\pi(s)\xi,\xi\rangle,\quad\forall s\in G,
\end{align*}
for some unitary representation $\pi:G\to\mathbf{U}(\mathcal{H})$ and $\xi\in\mathcal{H}$; see \cite[Theorem C.4.10]{BedlHVa}. Therefore the algebra $B(G)$ is the linear span of the positive definite functions on $G$.

\subsection{The Fourier algebra}
Let $G$ be a locally compact group endowed with a left Haar measure. Recall the definition of the left regular representation $\lambda:G\to\mathbf{U}(L^2(G))$ from \eqref{left_reg_rep}. We define $A(G)$ as the closed subspace of $B(G)$ spanned by all the coefficients of $\lambda$. In other words,
\begin{align*}
A(G)=\overline{\operatorname{span}}\{f\ast\check{g}\ \mid\ f,g\in L^2(G)\}\subseteq B(G),
\end{align*}
where $\check{g}(s)=g(s^{-1})$. This is the rigorous definition of $A(G)$ as a subspace of $B(G)$; however, it follows from \cite[Theorem 2.4.3]{KanLau}, \cite[Proposition 2.3.3]{KanLau} and \cite[Theorem 2.1.11]{KanLau} that
\begin{align*}
A(G)=\{f\ast\check{g}\ \mid\ f,g\in L^2(G)\},
\end{align*}
and that it is a closed ideal of $B(G)$. Moreover, the restriction of the norm of $B(G)$ to $A(G)$ can be calculated as follows:
\begin{align*}
\|a\|_{A(G)}=\inf\left\{\|f\|_2\|g\|_2\ \mid\ a=f\ast\check{g},\ f,g\in L^2(G)\right\},\quad\forall a\in A(G).
\end{align*}
Observe that $A(G)$ is a subalgebra of $C_0(G)$, the algebra of continuous functions on $G$ that vanish at infinity, and
\begin{align*}
\|a\|_\infty\leq\|a\|_{A(G)},\quad\forall a\in A(G).
\end{align*}

Recall that the group von Neumann algebra $LG$ is also defined from the left regular representation, and that being a von Neumann algebra, it is a dual Banach space. It turns out that its predual space is the Fourier algebra; see \cite[Theorem 2.3.9]{KanLau}. More precisely, we have $A(G)^*=LG$ for the duality
\begin{align*}
\langle\lambda(f),a\rangle=\int_G f(s)a(s)\, ds,\quad\forall f\in L^1(G),\ \forall a\in A(G).
\end{align*}

The following theorem of Leptin \cite{Lep} is the starting point for the definition of weak amenability. Recall that a locally compact group is said to be amenable if there is $\phi\in L^\infty(G)^*$ that is positive ($\phi(f)\geq 0$ for every $f\geq 0$), left invariant ($\phi(\delta_t\ast f)=\phi(f)$) and such that $\phi(1)=1$.

\begin{thm}[Leptin]\label{Thm_Lep}
A locally compact group $G$ is amenable if and only if its Fourier algebra $A(G)$ has a bounded approximate identity, i.e. there is a net $(a_i)$ in $A(G)$ such that
\begin{align*}
\|a_i\|_{A(G)}\leq 1,
\end{align*}
and
\begin{align*}
\|a_ia-a\|_{A(G)}\to 0,\quad\forall a\in A(G).
\end{align*}
\end{thm}

\begin{rmk}
For the reader who is familiar with the characterisation of amenability in terms of F\o{}lner sequences, if $\Gamma$ is an amenable, countable, discrete group, and $(F_n)$ is a F\o{}lner sequence in $\Gamma$, then
\begin{align*}
a_n(s)=\frac{|sF_n\cap F_n|}{|F_n|},\quad\forall s\in\Gamma,
\end{align*}
defines a bounded approximate identity in $A(\Gamma)$. Moreover, $(a_n)$ is positive definite and finitely supported.
\end{rmk}

\subsection{Fourier multipliers}\label{Ssec_Fourier_mult}
Recall that $A(G)$ is an ideal in $B(G)$. In other words, the multiplication by an element of $B(G)$ is a well defined map on $A(G)$. This leads to the following definition. We say that $\varphi:G\to\C$ is a multiplier of $A(G)$ if the linear map
\begin{align*}
m_\varphi: a\in A(G)\longmapsto \varphi a\in A(G)
\end{align*}
is well defined. Here $\varphi a$ means the pointwise product of $\varphi$ and $a$. An application of the closed graph theorem shows that, if $m_\varphi:A(G)\to A(G)$ is well defined, then it is bounded. We define $MA(G)$ as the space of all multipliers of $A(G)$, and we endow it with the norm
\begin{align*}
\|\varphi\|_{MA(G)}=\|m_\varphi\|_{\mathbf{B}(A(G))},\quad\forall \varphi\in MA(G),
\end{align*}
for which it becomes a Banach algebra. Moreover, by (the proof of) \cite[Proposition 2.3.2]{KanLau}, for every $t\in G$ and $\varepsilon>0$, there is a compact neighbourhood $K$ of $t$, and $a\in A(G)$ such that
\begin{align*}
\|a\|_{A(G)}\leq 1+\varepsilon,\qquad 0\leq a\leq 1,\qquad a|_K=1.
\end{align*}
This shows that every $\varphi\in MA(G)$ is necessarily continuous at $t$, and
\begin{align*}
|\varphi(t)|\leq \|m_\varphi a\|_\infty \leq \|m_\varphi a\|_{A(G)} \leq \|\varphi\|_{MA(G)}(1+\varepsilon).
\end{align*}
We conclude that $MA(G)$ is contained in $C_b(G)$ and
\begin{align*}
\|\varphi\|_\infty\leq\|\varphi\|_{MA(G)},\quad\forall \varphi\in MA(G).
\end{align*}

\begin{rmk}
Going back to the point of view of Fourier analysis on $\T$, a Fourier multiplier $\varphi\in MA(\Z)$ defines a linear map
\begin{align*}
\sum_{n\in\Z}\alpha_nz^n \longmapsto \sum_{n\in\Z}\varphi(n)\alpha_nz^n,
\end{align*}
that extends to a bounded map $L^1(\T)\to L^1(\T)$. In other words, $\varphi$ acts on $L^1(\T)$ by multiplying the Fourier coefficients.
\end{rmk}

Now let $M_\varphi:LG\to LG$ be the dual map of $m_\varphi:A(G)\to A(G)$. A simple calculation shows that
\begin{align}\label{mult_LG}
M_\varphi\lambda(f)=\lambda(\varphi f),\quad\forall f\in L^1(G).
\end{align}
Moreover, since it is a dual map, $M_\varphi$ is automatically weak*-weak*-continuous. We say that $\varphi\in MA(G)$ is a completely bounded multiplier of $A(G)$ if $M_\varphi$ is a completely bounded map on $LG$. We define $M_0A(G)$ as the space of all completely bounded multipliers of $A(G)$, and we endow it with the norm
\begin{align*}
\|\varphi\|_{M_0(G)}=\|M_\varphi\|_{\mathrm{cb}},\quad\forall\varphi\in M_0A(G),
\end{align*}
for which it becomes a Banach algebra.

\begin{rmk}
The algebra $M_0A(G)$ may also be defined using the operator space structure of $A(G)$ and its associated completely bounded maps; however, we will not follow that approach here. Instead, we refer the interested reader to \cite[\S 2.5]{Pis2}.
\end{rmk}

The following theorem is an essential piece of this theory. It provides a characterisation of $M_0(G)$ that is very useful for concrete computations. It was proved by Bo\.z{}ejko and Fendler in \cite{BozFen}; see also \cite{Jol}.

\begin{thm}[Bo\.z{}ejko--Fendler]\label{Thm_Boz-Fen}
Let $G$ be a locally compact group, $\varphi:G\to\C$, and $C>0$. The following are equivalent:
\begin{itemize}
\item[(i)] The function $\varphi$ is an element of $M_0A(G)$ with $\|\varphi\|_{M_0A(G)}\leq C$.
\item[(ii)] There is a Hilbert space $\mathcal{H}$, and continuous bounded maps $\xi,\eta:G\to\mathcal{H}$ such that
\begin{align*}
\varphi(s^{-1}t)=\langle\xi(t),\eta(s)\rangle,\quad\forall s,t\in G,
\end{align*}
and
\begin{align*}
\left(\sup_{t\in G}\|\xi(t)\|\right)\left(\sup_{s\in G}\|\eta(s)\|\right)\leq C.
\end{align*}
\end{itemize}
\end{thm}

\begin{rmk}
The functions described in Theorem \ref{Thm_Boz-Fen}.(ii) are called Herz--Schur multipliers; see Section \ref{Ssec_HS_mult}. Theorem \ref{Thm_Boz-Fen} says that the algebra of completely bounded multipliers of $A(G)$ and the algebra of Herz--Schur multipliers on $G$ coincide isometrically.
\end{rmk}

By Thoerem \ref{Thm_Boz-Fen}, it is quite clear that $B(G)$ is contained in $M_0(G)$. Hence, for every locally compact group, we get contractive inclusions
\begin{align*}
A(G)\hookrightarrow B(G)\hookrightarrow M_0A(G)\hookrightarrow MA(G)\hookrightarrow C_b(G).
\end{align*}
In general, none of these inclusions are surjective. However, for $G=\Z$, we have
\begin{align*}
B(\Z)=M_0A(\Z)=MA(\Z)=\left\{\hat{\mu}\ \mid\ \mu\in M(\T)\right\}.
\end{align*}
As a matter of fact, the equality $MA(G)=B(G)$ characterises amenability; see \cite[Theorem 5.1.10]{KanLau}.

\section{{\bf Weak amenability and the Cowling--Haagerup constant}}\label{Sec_wa}

In this section, we give the definition of weak amenability and the Cowling--Haagerup constant, and we relate them to the approximation properties discussed in Section \ref{Sec_OA_CBAP}. In order to motivate all this, we begin with a few words about amenability.

\subsection{Amenability}
By Theorem \ref{Thm_Lep}, amenability may be seen as an approximation property of the Fourier algebra. It turns out that it can also be characterised in terms of the approximation properties introduced in Section \ref{Sec_OA_CBAP}; see \cite[Theorem 4.2]{Lan} and \cite[Theorem 5.11]{EffLan}.

\begin{thm}[Lance, Effros--Lance]
Let $\Gamma$ be a discrete group. The following are equivalent:
\begin{itemize}
\item[(i)] The group $\Gamma$ is amenable.
\item[(ii)] The $\mathbf{C}^*$-algebra $C_\lambda^*(\Gamma)$ has the CPAP.
\item[(iii)] The von Neumann algebra $L\Gamma$ has the w*-CPAP.
\end{itemize}
\end{thm}

We will see that a similar characterisation holds for weak amenability.

\subsection{Weak amenability}
In order to define weak amenability, and to see that it really is a weakening of amenability, we start with the following corollary of Theorem \ref{Thm_Lep}.

\begin{cor}\label{Cor_Lep}
Let $G$ be an amenable, locally compact group. There exists a net $(a_i)$ in $A(G)$ such that $\|a_i\|_{A(G)}\leq 1$ and such that $a_i$ converges to $1$ uniformly on compact subsets of $G$.
\end{cor}
\begin{proof}
Let $(a_i)$ be the net given by Theorem \ref{Thm_Lep}. Now let $K$ be a compact subset of $G$. By \cite[Proposition 2.3.2]{KanLau}, there is $a\in A(G)$ such that $a|_K=1$. Hence
\begin{align*}
\sup_{t\in K}|a_i(s)-1|\leq\|a_ia-a\|_\infty\leq\|a_ia-a\|_{A(G)}\to 0.
\end{align*}
\end{proof}

We say that a locally compact group $G$ is weakly amenable if there is a constant $C\geq 1$, and a net $(a_i)$ in $A(G)$ such that
\begin{align*}
\|a_i\|_{M_0A(G)}\leq C,
\end{align*}
and such that $a_i$ converges to $1$ uniformly on compact subsets of $G$. The Cowling--Haagerup constant $\boldsymbol\Lambda(G)$ is defined as the infimum of all $C\geq 1$ such that the condition above holds. If $G$ is not weakly amenable, we simply set $\boldsymbol\Lambda(G)=\infty$.

We immediately obtain the following as a consequence of Corollary \ref{Cor_Lep}.

\begin{prop}\label{Prop_amen_wa}
Let $G$ be an amenable group. Then $G$ is weakly amenable with $\boldsymbol\Lambda(G)=1$.
\end{prop}

From the proposition above, one might feel tempted to suggest that the condition $\boldsymbol\Lambda(G)=1$ perhaps characterises amenability; however, we will see that this is far from being the case. As a matter of fact, free groups, which are the prototypical examples of nonamenable groups, satisfy $\boldsymbol\Lambda(\F_n)=1$; see e.g. Theorem \ref{Thm_Szwarc}. Observe that, in this case, the approximate identity $(\varphi_i)$ will be unbounded in the norm of $A(\F_n)$, but it will still satisfy $\|\varphi_i\|_{M_0A(\F_n)}\leq 1$.

As already mentioned, the main motivation behind the definition of weak amenability is the study of operator algebras. The link is given by the following theorem; see \cite[Proposition 6.1]{CowHaa} and \cite[Theorem 2.6]{Haa2}. We include its proof since it connects everything that we have discussed so far, and we believe it will help the reader get a more global understanding of these topics.

\begin{thm}[Haagerup, Cowling--Haagerup]\label{Thm_wa_CBAP}
Let $\Gamma$ be a discrete group. The following are equivalent:
\begin{itemize}
\item[(i)] The group $\Gamma$ is weakly amenable.
\item[(ii)] The $\mathbf{C}^*$-algebra $C_\lambda^*(\Gamma)$ has the CBAP.
\item[(iii)] The von Neumann algebra $L\Gamma$ has the w*-CBAP.
\end{itemize}
Moreover,
\begin{align*}
\boldsymbol\Lambda(\Gamma)=\boldsymbol\Lambda_{\mathrm{cb}}(C_\lambda^*(\Gamma))=\boldsymbol\Lambda_\mathrm{w^*cb}(L\Gamma).
\end{align*}
\end{thm}
\begin{proof}
Assume first that $\Gamma$ is weakly amenable, let $C>\boldsymbol\Lambda(\Gamma)$, and let $(\varphi_i)$ be an approximate identity in $A(\Gamma)$ such that $\|\varphi_i\|_{M_0A(\Gamma)}\leq C$. Since $\C[\Gamma]$ is dense in $A(\Gamma)$, and the inclusion $A(\Gamma)\hookrightarrow M_0A(\Gamma)$ is contractive, we may assume that $\varphi_i$ is finitely supported for all $i$. Now let $M_{\varphi_i}:L\Gamma\to L\Gamma$ be the weak*-continuous completely bounded map defined as in \eqref{mult_LG}. By definition
\begin{align*}
\|M_{\varphi_i}\|_{\mathrm{cb}}=\|\varphi_i\|_{M_0A(\Gamma)}\leq C.
\end{align*}
Moreover, since $\varphi_i$ is finitely supported, $M_{\varphi_i}$ has finite rank. More precisely, the range of $M_{\varphi_i}$ is the linear span of
\begin{align*}
\left\{\lambda(s)\ \mid\ s\in\operatorname{supp}(\varphi_i)\right\}.
\end{align*}
On the other hand, again by density,
\begin{align*}
\|\varphi_ia-a\|_{A(\Gamma)}\to 0,\quad\forall a\in A(\Gamma).
\end{align*}
Hence, for every $x\in L\Gamma$ and every $a\in A(\Gamma)$,
\begin{align*}
|\langle M_{\varphi_i}x-x,a\rangle|=|\langle x,(\varphi_i-1)a\rangle|\leq \|x\|_{L\Gamma}\|\varphi_ia-a\|_{A(\Gamma)}\to 0,
\end{align*}
which shows that $L\Gamma$ has the w*-CBAP and $\boldsymbol\Lambda_\mathrm{w^*cb}(L\Gamma)\leq C$. Now observe that, since $M_{\varphi_i}$ is a bounded map on $L\Gamma$ such that $M_{\varphi_i}\lambda(\ell^1(\Gamma))\subseteq \lambda(\ell^1(\Gamma))$, it restricts to a bounded map
\begin{align*}
\overline{M}_{\varphi_i}:C_\lambda^*(\Gamma)\to C_\lambda^*(\Gamma),
\end{align*}
which has finite rank too. Now let $x\in C_\lambda^*(\Gamma)$, $\varepsilon>0$, and $f\in\C[\Gamma]$ such that
\begin{align*}
\|x-\lambda(f)\|<\varepsilon.
\end{align*}
We have
\begin{align*}
\|\overline{M}_{\varphi_i}x-x\| &\leq \|\overline{M}_{\varphi_i}x-\overline{M}_{\varphi_i}\lambda(f)\| 
+ \|\overline{M}_{\varphi_i}\lambda(f)-\lambda(f)\| + \|\lambda(f)-x\|\\
&\leq (C+1)\varepsilon + \|\overline{M}_{\varphi_i}\lambda(f)-\lambda(f)\|.
\end{align*}
Since $\varphi_i$ converges to $1$ uniformly on the support of $f$,
\begin{align*}
\|\overline{M}_{\varphi_i}\lambda(f)-\lambda(f)\|&\leq \sum_{t\in\mathrm{supp}(f)}|\varphi_i(t)-1| |f(t)| \|\lambda(t)\|\\
&\leq \left(\sup_{t\in\mathrm{supp}(f)}|\varphi_i(t)-1|\right)\|f\|_1 \to 0.
\end{align*}
And since $\varepsilon$ was arbitrary, we conclude that
\begin{align*}
\|\overline{M}_{\varphi_i}x-x\|\to 0,\quad\forall x\in C_\lambda^*(\Gamma),
\end{align*}
which shows that $C_\lambda^*(\Gamma)$ has the CBAP and $\boldsymbol\Lambda_{\mathrm{cb}}(C_\lambda^*(\Gamma))\leq C$. We conclude that
\begin{align*}
\boldsymbol\Lambda_\mathrm{w^*cb}(L\Gamma)\leq \boldsymbol\Lambda(\Gamma)\qquad\text{and}\qquad \boldsymbol\Lambda_{\mathrm{cb}}(C_\lambda^*(\Gamma))\leq \boldsymbol\Lambda(\Gamma).
\end{align*}
Now assume that $C_\lambda^*(\Gamma)$ has the CBAP and take $C>\boldsymbol\Lambda_{\mathrm{cb}}(C_\lambda^*(\Gamma))$. Let $(T_i)$ be a net of finite rank maps on $C_\lambda^*(\Gamma)$ such that $\|T_i\|_{\mathrm{cb}}\leq C$ and
\begin{align*}
\|T_i x-x\|\to 0,\quad\forall x\in C_\lambda^*(\Gamma).
\end{align*}
Let $\tau\in C_\lambda^*(\Gamma)^*$ be defined by
\begin{align*}
\tau(x)=\langle x\delta_e,\delta_e\rangle,\quad\forall x\in C_\lambda^*(\Gamma),
\end{align*}
where $\delta_e\in\ell^2(\Gamma)$ denotes the Dirac mass at the identity element $e\in\Gamma$. Observe that
\begin{align*}
\tau(\lambda(f))=f(e),\quad\forall f\in\ell^1(\Gamma).
\end{align*}
For each $i$, we define $\varphi_i:\Gamma\to\C$ by
\begin{align*}
\varphi_i(t)=\tau(\lambda(t)^*T_i\lambda(t)),\quad\forall t\in\Gamma.
\end{align*}
Then, by \cite[Lemma 2.5]{Haa2}, $\varphi_i$ belongs to $\ell^2(\Gamma)\subseteq A(\Gamma)$, and
\begin{align*}
\|\varphi_i\|_{M_0A(\Gamma)}\leq \|T_i\|_{\mathrm{cb}}\leq C.
\end{align*}
Moreover, for every $t\in\Gamma$,
\begin{align*}
\varphi_i(t)=\langle T_i\delta_t,\delta_t\rangle\to 1,
\end{align*}
which means that $\varphi_i\to 1$ uniformly on finite sets. Therefore $\Gamma$ is weakly amenable and $\boldsymbol\Lambda(\Gamma)\leq C$. This gives us the implication (ii)$\implies$(i). The proof of (iii)$\implies$(i) follows analogously. We get
\begin{align*}
\boldsymbol\Lambda(\Gamma)\leq\boldsymbol\Lambda_{\mathrm{cb}}(C_\lambda^*(\Gamma)) \qquad\text{and}\qquad \boldsymbol\Lambda(\Gamma)\leq\boldsymbol\Lambda_\mathrm{w^*cb}(L\Gamma).
\end{align*}
\end{proof}

\subsection{Stability properties}
We now discuss the stability of weak amenability under basic group constructions and relations, as well as the behaviour of the Cowling--Haagerup constant. All the results presented here and their proofs can be found in \cite{CowHaa}.

\begin{prop}
Let $G$ be a locally compact group, and let $H$ be a closed subgroup of $G$. Then
\begin{align*}
\boldsymbol\Lambda(H)\leq\boldsymbol\Lambda(G).
\end{align*}
\end{prop}

\begin{prop}
Let $G$ and $H$ be locally compact groups. Then
\begin{align*}
\boldsymbol\Lambda(G\times H)=\boldsymbol\Lambda(G)\boldsymbol\Lambda(H).
\end{align*}
\end{prop}

\begin{prop}
Let $G$ be a locally compact group, and let $K$ be a compact normal subgroup of $G$. Then
\begin{align*}
\boldsymbol\Lambda(G/K)=\boldsymbol\Lambda(G).
\end{align*}
\end{prop}

\begin{prop}
Let $G$ be a locally compact group, and let $(G_i)$ be the net of compactly generated open subgroups of $G$ ordered by inclusion. Then
\begin{align*}
\boldsymbol\Lambda(G)=\lim_i\boldsymbol\Lambda(G_i).
\end{align*}
\end{prop}

Let $G$ be a locally compact group. A lattice in $G$ is a closed discrete subgroup $\Gamma$ such that the quotient $G/\Gamma$ admits a regular, $G$-invariant probability measure. Some examples of lattices are $\Z^n$ inside $\R^n$, and $\operatorname{SL}(n,\Z)$ inside $\operatorname{SL}(n,\R)$. The following result will be essential for studying lattices in Lie groups.

\begin{prop}\label{Prop_CH_lattice}
Let $G$ be a locally compact group, and let $\Gamma$ be a lattice in $G$. Then
\begin{align*}
\boldsymbol\Lambda(\Gamma)=\boldsymbol\Lambda(G).
\end{align*}
\end{prop}

\begin{rmk}
Weak amenability is not stable under group extensions. In \cite{Haa2}, Haagerup showed that $\operatorname{SL}(2,\R)\ltimes\R^2$ is not weakly amenable; however, both $\operatorname{SL}(2,\R)$ and $\R^2$ are weakly amenable with constant $1$; see Theorem \ref{Thm_Lie} and Proposition \ref{Prop_amen_wa}. %By Proposition \ref{Prop_CH_lattice}, the same holds for $\operatorname{SL}(2,\Z)\ltimes\Z^2$. Afterwards, Ozawa \cite{Oza2} gave a direct proof that $\operatorname{SL}(2,\Z)\ltimes\Z^2$ is not weakly amenable.
\end{rmk}

\section{{\bf Weak amenability for Lie groups}}\label{Sec_Lie}

Since Theorem \ref{Thm_wa_CBAP} is the main motivation for studying weak amenability, one might argue that this property is only relevant for discrete groups, which would considerably simplify the presentation in this introductory text. However, there is a good reason for working in the realm of locally compact groups from the beginning: Lie groups. The analytic tools provided by the theory of Lie groups, together with Proposition \ref{Prop_CH_lattice} allow us to draw information about the operator algebras associated to their lattices. Furthermore, before receiving the name \textit{weak amenability}, the first examples of groups for which this property was studied were the Lie groups $\operatorname{SU}(n,1)$ in \cite{Cow}.

After an enormous body of work that spanned over almost four decades, weak amenability for Lie groups is very well understood nowadays, as well as the values of their Cowling--Haagerup constants. We will summarise all these results in Theorems \ref{Thm_Lie}, \ref{Thm_Lie2} and \ref{Thm_Lie3} below, following the presentations of \cite{Knu3} and \cite{Knu}. 

Before stating these results, we briefly discuss the notions of real rank and local isomorphism; for more details on Lie groups and their structure, we refer the reader to \cite{Hel} and \cite{Kna}. Let $G$ be a simple Lie group with Lie algebra $\mathfrak{g}$. The Cartan decomposition of $\mathfrak{g}$ is
\begin{align*}
\mathfrak{g}=\mathfrak{k}+\mathfrak{p},
\end{align*}
where $\mathfrak{k}$ and $\mathfrak{p}$ are the eigenspaces for the Cartan involution $\theta:\mathfrak{g}\to\mathfrak{g}$, associated to the eigenvalues $1$ and $-1$ respectively; see \cite[\S VI.2]{Kna}. The real rank of $G$ is the dimension of a maximal abelian subspace of $\mathfrak{p}$. We will use the notation $G\approx H$ to say that $G$ is locally isomorphic to $H$, meaning that their Lie algebras are isomorphic.

\begin{thm}\label{Thm_Lie}
Let $G$ be a connected simple Lie group. Then $G$ is weakly amenable if and only if it has real rank $0$ or $1$. Furthermore,
\begin{align*}
\boldsymbol\Lambda(G)=\begin{cases}
1 & \text{if }\ G \text{ has real rank } 0,\\
1 & \text{if }\ G\approx\operatorname{SO}(n,1),\ n\geq 2,\\
1 & \text{if }\ G\approx\operatorname{SU}(n,1),\ n\geq 2,\\
2n-1 & \text{if }\ G\approx\operatorname{Sp}(n,1),\ n\geq 2,\\
21 & \text{if }\ G\approx\operatorname{F}_{4,-20}.
\end{cases}
\end{align*}
\end{thm}

By the classification of simple real Lie algebras (see e.g. \cite[Theorem 6.105]{Kna}), any connected simple Lie group of real rank $1$ is locally isomorphic to either $\operatorname{SO}(n,1)$, $\operatorname{SU}(n,1)$, $\operatorname{Sp}(n,1)$ or $\operatorname{F}_{4,-20}$. Hence Theorem \ref{Thm_Lie} completely determines the value of $\boldsymbol\Lambda(G)$ for any connected simple Lie group $G$. We will not give any details on the proof of this very deep theorem; we will simply mention that an essential ingredient is the study of uniformly bounded representations, in a similar spirit as in Section \ref{Ssec_wa_ubr}.

Now we briefly describe the history of how Theorem \ref{Thm_Lie} came to be. As mentioned above, in \cite{Cow}, Cowling showed that $\boldsymbol\Lambda(\operatorname{SU}(n,1))=1$ without using that language. %Even more, Theorem \ref{Thm_Boz-Fen} had not yet been proved at that time\todo{check}, so his result was not expressed in terms of completely bounded Fourier multipliers, but with Herz--Schur multipliers instead, which were later shown to be the same objects.

In \cite{deCHaa}, de Canni\`ere and Haagerup added $\operatorname{SO}_0(n,1)$ and its closed subgroups to the list. In particular, they showed that all free groups $\F_N$ ($2\leq N\leq\infty$) are weakly amenable and $\boldsymbol\Lambda(\F_N)=1$. Again, the Cowling--Haagerup constant had not yet been introduced, so they did not state their results that way.

In \cite{Haa2}, Haagerup showed that, for every simple Lie group $G$ of real rank at least $2$ and finite centre, $\boldsymbol\Lambda(G)=\infty$. Although that paper was published only in 2016, a highly cited preprint had been circulating amongst experts since 1986.

Weak amenability and the Cowling--Haagerup constant were finally introduced in \cite{CowHaa}. In that same article, the authors showed that every simple Lie group with real rank $1$ and finite centre is weakly amenable, giving the exact value of the Cowling--Haagerup constant in each case. In particular, they computed $\boldsymbol\Lambda(\operatorname{Sp}(n,1))$ and $\boldsymbol\Lambda(\operatorname{F}_{4,-20})$, providing the first examples of weakly amenable groups with Cowling--Haagerup constant strictly greater than $1$. A very important corollary of their results is the following.

\begin{cor}[Cowling--Haagerup]
Let $n,m\geq 2$ with $n\neq m$. Let $\Gamma$ and $\Lambda$ be lattices in $\operatorname{Sp}(n,1)$ and $\operatorname{Sp}(m,1)$ respectively. Then $C_\lambda^*(\Gamma)$ and $C_\lambda^*(\Lambda)$ are not isomorphic as $\mathbf{C}^*$-algebras, and $L\Gamma$ and $L\Lambda$ are not isomorphic as von Neumann algebras.
\end{cor}

In \cite{Han}, Hansen was able to remove the finite centre condition from the main result of \cite{CowHaa} by studying the universal covering group of $\operatorname{SU}(n,1)$. 

Subsequently, Dorofaeff showed in \cite{Dor,Dor2} that the finite centre assumption can also be removed from the main result of \cite{Haa2}. In other words, every simple Lie group $G$ of real rank at least $2$ satisfies $\boldsymbol\Lambda(G)=\infty$. All these results together yield Theorem \ref{Thm_Lie}.

For semisimple Lie groups, the situation is also very well understood.

\begin{thm}\label{Thm_Lie2}
Let $G$ be a connected semisimple Lie group. Then $G$ is locally isomorphic to a direct product $S_1\times\cdots\times S_n$ of connected simple Lie groups $S_1,\ldots,S_n$. Moreover, $G$ is weakly amenable if and only if each $S_i$ is weakly amenable, and
\begin{align*}
\boldsymbol\Lambda(G)=\prod_{i=1}^n \boldsymbol\Lambda(S_i).
\end{align*}
\end{thm}

In the non-semisimple case, there is not a complete description, but still much is known. The following result is expressed in terms of the Levi decomposition of real Lie algebras; see e.g. \cite[\S B.1]{Kna}.

\begin{thm}\label{Thm_Lie3}
Let $G$ be a connected Lie group, and let $\mathfrak{g}=\mathfrak{s}\ltimes\mathfrak{r}$ be a Levi decomposition of its Lie algebra such that $\mathfrak{s}$ decomposes as a sum of simple ideals $\mathfrak{s}_1\oplus\cdots\oplus\mathfrak{s}_n$. Let $S$ be the semisimple Lie group associated to $\mathfrak{s}$. Assume that $G$ is simply connected or that $S$ has finite centre. Then the following are equivalent:
\begin{itemize}
\item[a)] $G$ is weakly amenable.
\item[b)] For each $i\in\{1,\ldots,n\}$, one of the following holds:
\begin{itemize}
\item[$\bullet$] $\mathfrak{s}_i$ has real rank $0$.
\item[$\bullet$] $\mathfrak{s}_i$ has real rank $1$ and $[\mathfrak{s}_i,\mathfrak{r}]=0$.
\end{itemize}
\end{itemize}
In that case, if $S_i$ denotes the connected simple Lie group associated to $\mathfrak{s}_i$,
\begin{align*}
\boldsymbol\Lambda(G)=\prod_{i=1}^n \boldsymbol\Lambda(S_i).
\end{align*}
\end{thm}

In \cite{CDSW}, Cowling, Dorofaeff, Seeger and Wright proved Theorem \ref{Thm_Lie3} in the case when $S$ has finite centre, which includes all real algebraic Lie groups. The case when $G$ is simply connected was proved by Knudby in \cite{Knu} following an approach developed by Ozawa \cite{Oza2}.

Finally, Theorem \ref{Thm_Lie2} was proved in full generality by Knudby in \cite{Knu3}. Observe that the particular cases when $G$ is simply connected or has finite centre can be deduced from Theorem \ref{Thm_Lie3}.

\begin{rmk}
Let $G$ be a Lie group as in any of the theorems above. Then, when $G$ is weakly amenable, $\boldsymbol\Lambda(G)$ is an odd natural number. It is not known whether the Cowling--Haagerup constant of a locally compact group can take values outside of $(2\N+1)\cup\{\infty\}$.
\end{rmk}

\section{{\bf Weak amenability from geometric group theory}}\label{Sec_ggt}

One of the consequences of Theorem \ref{Thm_Lie} is that free groups are weakly amenable because they can be realised as closed subgroups of $\operatorname{SO}(2,1)$. It turns out that this fact can also be derived in a more direct way from the geometry of trees. This idea was first exploited by Szwarc \cite{Szw} by means of uniformly bounded representations and by Bo\.{z}ejko and Picardello \cite{BozPic} using radial Schur multipliers. These strategies have become extremely fruitful in the study of weak amenability, and have been extended to several different contexts.

\subsection{Schur multipliers}
Let $X$ be a set, and let $\ell^2(X)$ be the Hilbert space of complex-valued, square-summable functions on $X$. Observe that every $T\in\mathbf{B}(\ell^2(X))$ is completely determined by its matrix coefficients:
\begin{align*}
T_{x,y}=\langle T\delta_y,\delta_x\rangle,\quad\forall x,y\in X.
\end{align*} 
We say that $\varphi:X\times X\to\C$ is a Schur multiplier on $X$ if the map
\begin{align*}
T=(T_{x,y})_{x,y\in X}\in\mathbf{B}(\ell^2(X)) \ \longmapsto\ (\varphi(x,y)T_{x,y})_{x,y\in X}\in\mathbf{B}(\ell^2(X))
\end{align*}
is well defined. In that case, thanks to the closed graph theorem, it defines a bounded linear map $M_\varphi:\mathbf{B}(\ell^2(X))\to\mathbf{B}(\ell^2(X))$.

The following result clarifies why these objects are relevant for weak amenability; for a proof, we refer the reader to \cite[Theorem 5.1]{Pis3}. This theorem may be attributed to Grothendieck, Gilbert and Haagerup; see the Notes and Remarks at the end of Chapter 5 of \cite{Pis3} for a more detailed discussion about this.

\begin{thm}\label{Thm_Schur_mult}
Let $X$ be a set, $\varphi:X\times X\to\C$, and $C\geq 0$. The following are equivalent:
\begin{itemize}
\item[(i)] The function $\varphi$ is a Schur multiplier and the associated map $M_\varphi:\mathbf{B}(\ell^2(X))\to\mathbf{B}(\ell^2(X))$ satisfies $\|M_\varphi\|\leq C$.
\item[(ii)] The function $\varphi$ is a Schur multiplier and the associated map $M_\varphi:\mathbf{B}(\ell^2(X))\to\mathbf{B}(\ell^2(X))$ is completely bounded with $\|M_\varphi\|_{\mathrm{cb}}\leq C$.
\item[(iii)] There is a Hilbert space $\mathcal{H}$ and bounded maps $\xi,\eta:X\to\mathcal{H}$ such that
\begin{align*}
\varphi(x,y)=\langle\xi(x),\eta(y)\rangle,\quad\forall x,y\in X,
\end{align*}
and
\begin{align*}
\left(\sup_{x\in X}\|\xi(x)\|\right)\left(\sup_{y\in X}\|\eta(y)\|\right)\leq C.
\end{align*}
\end{itemize}
\end{thm}

We let $V(X)$ denote the space of Schur multipliers on $X$, and we endow it with the norm
\begin{align*}
\|\varphi\|_{V(X)}=\|M_\varphi\|=\|M_\varphi\|_{\mathrm{cb}},\quad\forall \varphi\in V(X),
\end{align*}
for which it becomes a commutative Banach algebra for pointwise operations.

\subsection{Herz--Schur multipliers}\label{Ssec_HS_mult}
Now let $G$ be a (discrete) group. We say that $\varphi:G\to\C$ is a Herz--Schur multiplier if the function $\tilde{\varphi}: G\times G\to\C$ given by
\begin{align*}
\tilde{\varphi}(t,s)=\varphi(s^{-1}t),\quad\forall t,s\in G,
\end{align*}
belongs to $V(G)$. By Theorem \ref{Thm_Boz-Fen}, this is equivalent to being a completely bounded multiplier of $A(G)$. Moreover,
\begin{align*}
\|\varphi\|_{M_0A(G)}=\|\tilde{\varphi}\|_{V(G)}.
\end{align*}
If $G$ is a locally compact group, the maps $\xi,\eta$ in Theorem \ref{Thm_Boz-Fen} are assumed to be continuous; however, one can show that the continuity of $\varphi$ suffices to get a similar characterisation; see \cite[Theorem 3.2]{Haa2}.

As we saw in Section \ref{Ssec_Fourier_mult}, the coefficients of unitary representations of $G$ belong to $M_0A(G)$. More generally, we say that a representation $\pi:G\to\mathbf{B}(\mathcal{H})$ is uniformly bounded if
\begin{align*}
|\pi|=\sup_{t\in G}\|\pi(t)\|<\infty.
\end{align*}
If $\varphi:G\to\C$ is defined by
\begin{align*}
\varphi(t)=\langle\pi(t)\xi,\eta\rangle,\quad\forall t\in G,
\end{align*}
for some $\xi,\eta\in\mathcal{H}$, then $\varphi$ belongs to $M_0A(G)$ and
\begin{align*}
\|\varphi\|_{M_0A(G)}\leq|\pi|^2\|\xi\|\|\eta\|.
\end{align*}
This observation allows us to describe a very large class of multipliers. However, if $G$ contains a free subgroup, there are elements of $M_0A(G)$ that do not arise as coefficients of uniformly bounded representations; see \cite[Remark 3.2]{Pis4} and \cite[Lemma 1.2]{BozFen2}.

\subsection{Weak amenability via uniformly bounded representations}\label{Ssec_wa_ubr}
For a group acting on a tree, Szwarc constructed \cite{Szw} an analytic family of uniformly bounded representations with very interesting properties. One of the consequences of his construction is the following; see \cite[Theorem 6]{Szw}.

\begin{thm}[Szwarc]\label{Thm_Szwarc}
Let $G$ be a locally compact group acting on a locally finite tree $X$. Assume that there is a vertex $x_0\in X$ such that the stabiliser $G_{x_0}=\{t\in G\ \mid\ tx_0=x_0\}$ is compact. Then $G$ is weakly amenable and $\boldsymbol\Lambda(G)=1$.
\end{thm}

As a consequence, free groups are weakly amenable with Cowling--Haagerup constant $1$. The same holds for $\operatorname{SL}(2,\mathbb{Q}_p)$, were $\mathbb{Q}_p$ denotes the field of $p$-adic numbers for some prime $p$; see \cite[\S II.1]{Ser} or \cite[Theorem 10.2.1]{CohGel}. Similar constructions to that of \cite{Szw} were given in \cite{Pim}, \cite{PytSzw} and \cite{Val}.

Theorem \ref{Thm_Szwarc} has also been used to prove weak amenability of Baumslag--Solitar groups:
\begin{align*}
\operatorname{BS}(m,n)=\langle a,b\ \mid\ ba^mb^{-1}=a^n \rangle.
\end{align*}
The argument was developed by Gal and Januszkiewicz in \cite{GalJan}, and it consists in embedding $\operatorname{BS}(m,n)$ into a product of weakly amenable groups, where one of the factors is the automorphism group of a locally finite tree. The goal of \cite{GalJan} was to prove the Haagerup property for such groups, and weak amenability is nowhere mentioned in that paper. It was later observed in \cite[Example 1.4]{CorVal} that the same argument proves weak amenability.

\begin{thm}[Cornulier--Valette]\label{Thm_BS(m,n)}
Let $m,n\in\Z$. The Baumslag--Solitar group $\operatorname{BS}(m,n)$ is weakly amenable and $\boldsymbol\Lambda(\operatorname{BS}(m,n))=1$.
\end{thm}

We refer the reader to \cite[Corollary 1.7]{CorVal} for a more general result on generalised Baumslag--Solitar groups.

The idea of proving weak amenability via analytic families of representations was taken even further by Valette in \cite{Val2}. Let $G$ be a locally compact group, and let $e$ denote the identity element of $G$. Recall that a map $L:G\to\N$ is proper if $L^{-1}(\{n\})$ is compact for every $n\in\N$. Let $\D$ denote the open unit disc in $\C$.

\begin{thm}[Valette]\label{Thm_Valette_ubr}
Let $G$ be a locally compact group endowed with a continuous, proper map $L:G\to\N$ such that $L(e)=0$. Assume that there is an analytic family of uniformly bounded representations $(\pi_z)_{z\in\D}$ of $G$ on a Hilbert space $\mathcal{H}$ such that $\pi_t$ is unitary for all $t\in(0,1)$. Assume also that there is $\xi\in\mathcal{H}$ such that
\begin{align*}
z^{L(s)}=\langle\pi_z(s)\xi,\xi\rangle,\quad\forall z\in\D,\ \forall s\in G.
\end{align*}
Then $G$ is weakly amenable and $\boldsymbol\Lambda(G)=1$.
\end{thm}

This result was established in order to show that right-angled Coxeter groups are weakly amenable; see \cite{Jan}. Afterwards, Guentner and Higson \cite{GueHig} proved something much stronger. They showed that any group acting properly on a finite-dimensional $\operatorname{CAT}(0)$ cube complex satisfies the hypotheses of Theorem \ref{Thm_Valette_ubr}. We will not define $\operatorname{CAT}(0)$ cube complexes here, we only mention that a product of $N$ trees is an example of an $N$-dimensional $\operatorname{CAT}(0)$ cube complex; for details, we refer the reader to \cite[\S 2]{GueHig}, and to \cite[\S 2.A]{Cor} for examples of groups satisfying the hypotheses of Theorem \ref{Thm_GueHig} below. This result was proved independently by Mizuta using radial Schur multipliers; see Section \ref{Ssec_rad_Schur}.

\begin{thm}[Guentner--Higson, Mizuta]\label{Thm_GueHig}
Let $G$ be a countable group acting properly on a finite-dimensional $\operatorname{CAT}(0)$ cube complex. Then $G$ is weakly amenable and $\boldsymbol\Lambda(G)=1$.
\end{thm}

\begin{rmk}
The finite dimension hypothesis in Theorem \ref{Thm_GueHig} is essential. Infinite-dimensional complexes allow much more pathological behaviours; see Theorem \ref{Thm_Osa} and Theorem \ref{Thm_wa->ex}.
\end{rmk}

A similar construction of representations was given in \cite{CosMar} for mapping class groups; however, it is not clear whether it can be used to prove weak amenability. The following question remains open.

\begin{ques}
Are mapping class groups of compact punctured surfaces weakly amenable?
\end{ques}

\subsection{Weak amenability via radial Schur multipliers}\label{Ssec_rad_Schur}
Recall that (the set of vertices of) a connected graph $X$ may be viewed as a metric space when we endow it with the edge-path distance, i.e. the length of the shortest path joining each pair of vertices. For each $n\in\N$, we can define $\chi_n:X\times X\to\{0,1\}$ by
\begin{align}\label{chi_n}
\chi_n(x,y)=\begin{cases}
1, & \text{if } d(x,y)=n,\\
0, & \text{otherwise.}
\end{cases}
\end{align}
Observe that $\chi_n$ depends only on the distance on $X$. We will call this kind of functions radial. The following result is the starting point for the study of weak amenability via radial Schur multipliers; see \cite[Proposition 2.1]{BozPic}. Recall that a tree is a connected graph without cycles.

\begin{thm}[Bo\.{z}ejko--Picardello]\label{Thm_Boz-Pic}
Let $X$ be a tree, and let $\chi_n:X\times X\to\{0,1\}$ be defined as in \eqref{chi_n}. Then, for every $n\in\N$, $\chi_n$ is a Schur multiplier on $X$, and
\begin{align*}
\|\chi_n\|_{V(X)}\leq 2n,\quad\forall n\geq 1.
\end{align*}
\end{thm}

This result was proved in order to show that certain amalgamated free products are weakly amenable; see Theorem \ref{Thm_Boz-Pic2}. The polynomial growth of $\|\chi_n\|_{V(X)}$ can be combined with the fact that the kernel
\begin{align}\label{pd_ker_d}
(x,y)\in X\times X \longmapsto e^{-\lambda d(x,y)}
\end{align}
is positive definite for all $\lambda>0$, which allows one to show that there is a sequence of truncated maps
\begin{align*}
\varphi_n(x,y)=\sum_{j=0}^{k_n}\chi_j(x,y)e^{-j/n},\quad\forall x,y\in X,
\end{align*}
such that $\varphi_n$ converges to $1$ pointwise, and $\|\varphi_n\|_{V(X)}\to 1$; see \cite[\S 12.3]{BroOza} for details. Then, if $G$ is a group acting properly by isometries on $X$, fixing a vertex $x_0\in X$, one can consider the sequence of Herz--Schur multipliers
\begin{align*}
s\in G \longmapsto \varphi_n(sx_0,x_0),
\end{align*}
and show that $\boldsymbol\Lambda(G)=1$. In particular, this holds for free groups and $\operatorname{SL}(2,\mathbb{Q}_p)$.

Theorem \ref{Thm_Boz-Pic} has also been used to prove weak amenability of $\operatorname{GL}(2,F)$, where $F$ is a field. Similarly to how Theorem \ref{Thm_BS(m,n)} was obtained, it was first proved in \cite{GuHiWe} that $\operatorname{GL}(2,F)$ has the Haagerup property, and the argument was later adapted in \cite[Theorem 1.11]{KnuLi} for weak amenability.

\begin{thm}[Knudby--Li]
Let $F$ be a field. Then $\operatorname{GL}(2,F)$, as a discrete group, is weakly amenable and $\boldsymbol\Lambda(\operatorname{GL}(2,F))=1$.
\end{thm}

As a consequence of a similar result for commutative rings, they also obtained the following; see \cite[Corollary 3.12]{KnuLi}.

\begin{thm}[Knudby--Li]
Residually free groups are weakly amenable with Cowling--Haagerup constant $1$.
\end{thm}

As mentioned above, the idea of proving weak amenability using truncated exponential functions has been extended to much more general contexts. In \cite{Oza}, Ozawa showed that, although the kernels \eqref{pd_ker_d} are no longer positive definite on a general hyperbolic graph, they do define a uniformly bounded semigroup of Schur multipliers, provided that the graph has bounded geometry. This fact, combined with an estimate similar to Theorem \ref{Thm_Boz-Pic}, allows one to prove weak amenability in similar fashion; see \cite[Theorem 1]{Oza}. We refer the reader to \cite[\S 7]{Loh} for details on hyperbolic spaces and groups.

\begin{thm}[Ozawa]\label{Thm_Ozawa_wa}
Hyperbolic groups are weakly amenable.
\end{thm}

Theorem \ref{Thm_Ozawa_wa} says that $\boldsymbol\Lambda(\Gamma)$ is finite for every hyperbolic group $\Gamma$, but it does not provide an estimate on its value. If one examines the proof, one can find an upper bound in terms of the maximal degree and the hyperbolicity constant of the graph on which the group acts, but it is not possible to get a uniform bound for all hyperbolic groups. Indeed, uniform lattices in $\operatorname{Sp}(n,1)$ are hyperbolic for all $n\geq 2$, and we know from Theorem \ref{Thm_Lie} that such a lattice satisfies $\boldsymbol\Lambda(\Gamma)=2n-1$.

These ideas have also been extended to some relatively hyperbolic groups \cite{GuReTe}, and groups acting on products of hyperbolic graphs \cite{Ver}.

As mentioned earlier, radial Schur multipliers were also used to prove Theorem \ref{Thm_GueHig}. In \cite{Miz}, Mizuta extended Theorem \ref{Thm_Boz-Pic} to all finite-dimensional $\operatorname{CAT}(0)$ cube complexes. In this case, the growth of $\|\chi_n\|_{V(X)}$ is no longer linear, but it is still polynomial, and the degree depends on the dimension of the complex. Then the fact that the semigroup \eqref{pd_ker_d} is positive definite on any $\operatorname{CAT}(0)$ cube complex allows one to repeat the argument for trees in this more general setting.

\subsection{Free products}
As mentioned above, the main goal of Theorem \ref{Thm_Boz-Pic} was to prove weak amenability of amalgamated free products, by exploiting the intimate connection that exists between free products and trees; see \cite{Ser}. We now state the main result of \cite{BozPic}.

\begin{thm}[Bo\.{z}ejko--Picardello]\label{Thm_Boz-Pic2}
Let $(G_i)_{i\in I}$ be a family of amenable locally compact groups, and let $A$ be an open compact subgroup of $G_i$ for every $i\in I$. Then the amalgamated free product $G=\ast_A G_i$ is weakly amenable and $\boldsymbol\Lambda(G)=1$.
\end{thm}

In general, it is not known if weak amenability is preserved by free products, but some particular cases are known. The following was proved in  \cite[Theorem 4.13]{RicXu}.

\begin{thm}[Ricard--Xu]\label{Thm_RicXu}
Let $(G_i)_{i\in I}$ be a family of discrete groups such that $\boldsymbol\Lambda(G_i)=1$ for every $i\in I$. Then the free product $G=\ast_{i\in I} G_i$ is weakly amenable and $\boldsymbol\Lambda(G)=1$.
\end{thm}

This theorem was obtained as corollary of a more general result on free products of $\mathbf{C}^*$-algebras; see \cite[Proposition 4.11]{RicXu}. Afterwards, Reckwerdt \cite{Rec} gave another proof of Theorem \ref{Thm_RicXu}, extending it to graph products of groups.

\begin{thm}[Reckwerdt]\label{Thm_Rec}
Let $G=\ast_{v,\Gamma} G_v$ be the graph product of a family of groups $(G_v)_{v\in\Gamma}$, where $\Gamma$ is a finite graph. If $\boldsymbol\Lambda(G_v)=1$ for every $v\in\Gamma$, then $\boldsymbol\Lambda(G)=1$.
\end{thm}

More recently, a new proof of Theorem \ref{Thm_Rec} was given in \cite{Bor} by studying graph products of $\mathbf{C}^*$-algebras, along the lines of \cite{RicXu}.

Theorem \ref{Thm_RicXu} says that weak amenability with Cowling--Haagerup constant 1 is stable under free products, but the following question remains open.

\begin{ques}\label{Ques_freeprod}
Let $G,H$ be groups such that
\begin{align*}
1<\max\{\boldsymbol\Lambda(G),\boldsymbol\Lambda(H)\}<\infty.
\end{align*}
Can we conclude that $\boldsymbol\Lambda(G\ast H)<\infty$?
\end{ques}

In the particular case when $G$ is amenable and $H$ is hyperbolic, Question \ref{Ques_freeprod} has a positive answer; see \cite{Ver2}. See also \cite[Corollary 5.2]{GuReTe} for the case when $G$ has polynomial growth.

\section{{\bf Weak amenability and measured group theory}}\label{Sec_mgt}

Even though geometric tools have proved to be very useful to the study of weak amenability, this is not a geometric property. More precisely, one can find pairs of finitely generated groups which are quasi-isometric, but one is weakly amenable and the other one is not; see \cite{Car}. In sharp contrast with this, weak amenability behaves very well with respect to equivalence relations coming from measured group theory; we refer the reader to \cite{Fur} for an introduction to these topics.

Let $\Gamma, \Lambda$ be countable groups. We say that two measure preserving actions $\Gamma\curvearrowright(X,\mu)$, $\Lambda\curvearrowright(Y,\nu)$ on standard non-atomic probability spaces are orbit equivalent if there is a measure space isomorphism $T:(X,\mu)\to(Y,\nu)$ sending $\Gamma$-orbits onto $\Lambda$-orbits. The following was proved in \cite[Lemma 2.6]{CowZim} and \cite[Theorem 3.1]{CowZim}.

\begin{thm}[Cowling--Zimmer]
Let $\Gamma, \Lambda$ be countable groups admitting orbit equivalent actions as above. Then $\boldsymbol\Lambda(\Gamma)=\boldsymbol\Lambda(\Lambda)$.
\end{thm}

This result was extended by Jolissaint to measure equivalent groups in an unpublished manuscript; see \cite[Theorem 1.10]{Jol3} for a more general result. Two countable groups $\Gamma, \Lambda$ are measure equivalent if there is a measure preserving action of $\Gamma\times\Lambda$ on a standard measure space $(\Omega,\mu)$ such that each of the actions $\Gamma\curvearrowright\Omega$, $\Lambda\curvearrowright\Omega$ admits a finite-measure fundamental domain. In other words, there are measurable sets $X,Y\subset\Omega$ such that $\mu(X)<\infty$, $\mu(Y)<\infty$, and
\begin{align}\label{ME_coupling}
\Omega=\bigsqcup_{s\in\Gamma}sX=\bigsqcup_{t\in\Lambda}tY.
\end{align}
This is a weakening of orbit equivalence in the following sense. If $\Gamma$ and $\Lambda$ admit essentially free ergodic probability measure-preserving actions that are orbit equivalent, then one can find a measure equivalence coupling $\Gamma\times\Lambda\curvearrowright(\Omega,\mu)$ with one common fundamental domain for both individual actions; see the proof of \cite[Theorem 2.5]{Fur}.

\begin{thm}[Jolissaint]\label{Thm_Jol}
Let $\Gamma, \Lambda$ be measure equivalent groups. Then $\boldsymbol\Lambda(\Gamma)=\boldsymbol\Lambda(\Lambda)$.
\end{thm}

A simple proof of Theorem \ref{Thm_Jol} was given by Ozawa in \cite[\S 2]{Oza2}. One can also consider the more general setting in which only one of the fundamental domains has finite measure. More precisely, suppose that \eqref{ME_coupling} holds, but we only know that $X$ has finite measure. Then the same proof shows that $\boldsymbol\Lambda(\Lambda)\leq\boldsymbol\Lambda(\Gamma)$.

More recently, these results were generalised even further. A noncommutative version of measure equivalence, called von Neumann equivalence, was introduced in \cite{IsPeRu}. In this case, measure spaces are replaced by tracial von Neumann algebras, and fundamental domains are replaced by projections; we refer the reader to \cite{IsPeRu} for details. Measure equivalence implies von Neumann equivalence, but it is not known whether the converse holds. The following was proved in \cite[Theorem 1.1]{Ish}.

\begin{thm}[Ishan]
Let $\Gamma, \Lambda$ be von Neumann equivalent groups. Then $\boldsymbol\Lambda(\Gamma)=\boldsymbol\Lambda(\Lambda)$.
\end{thm}

Again, one may assume that only one of the fundamental domains has finite trace and obtain an inequality instead.

\section{{\bf Obstructions to weak amenability}}\label{Sec_obstr}

In this section, we focus on groups that are known not to be weakly amenable. The first examples, as we already saw in Section \ref{Sec_Lie}, come from higher rank Lie groups and their lattices. In particular, in \cite{Haa2}, Haagerup showed that
\begin{align}\label{Lambda(SL(2,Z)Z2)}
\boldsymbol\Lambda(\operatorname{SL}(2,\R)\ltimes\R^2)=\boldsymbol\Lambda(\operatorname{SL}(2,\Z)\ltimes\Z^2)=\infty.
\end{align}
This was generalised by Ozawa \cite{Oza2} in the following way; see also \cite[Proposition 2.11]{OzaPop}.

\begin{thm}[Ozawa]\label{Thm_Oza_nonwa}
Let $G$ be a weakly amenable group, and let $N$ be a normal, closed, amenable subgroup of $G$. Then there is a $(G\ltimes N)$-invariant mean on $L^\infty(N)$, where $(G\ltimes N)$ acts on $N$ by $(s,t)\cdot x=stxs^{-1}$.
\end{thm}

Applying this result to $G=\operatorname{SL}(2,\Z)\ltimes\Z^2$ and $N=\Z^2$, one recovers \eqref{Lambda(SL(2,Z)Z2)}. Moreover, it can also be applied to wreath products. Let $\Gamma$ and $\Lambda$ be discrete groups, and let $\bigoplus_\Gamma\Lambda$ denote the direct sum of copies of $\Lambda$ indexed by $\Gamma$. The wreath product $\Lambda\wr\Gamma$ is defined as
\begin{align*}
\Lambda\wr\Gamma=\Gamma\ltimes\bigoplus_\Gamma\Lambda,
\end{align*}
where $\Gamma$ acts on $\bigoplus_\Gamma\Lambda$ by shifts.

\begin{cor}[Ozawa]\label{Cor_Oza}
Let $\Gamma$ and $\Lambda$ be discrete groups such that $\Gamma$ is not amenable and $\Lambda$ is not the trivial group. Then $\Lambda\wr\Gamma$ is not weakly amenable.
\end{cor}

As we will see in Section \ref{Sec_AP_HK}, there is an even weaker property than weak amenability, with important applications to operator algebras: the AP of Haagerup and Kraus. All the examples that we shall discuss now fail this more general property, but since we are primarily interested in weak amenability here, we will not mention this any further. All the references, however, deal with the AP and its operator algebraic counterpart.

We begin by discussing exact groups; we refer the reader to Chapter 5 of \cite{BroOza} for a detailed presentation of these topics. We will say that a discrete group $\Gamma$ is exact if its reduced $\mathbf{C}^*$-algebra $C_\lambda^*(\Gamma)$ is exact, meaning that, for every $\mathbf{C}^*$-algebra $\mathcal{A}$ and every ideal $\mathcal{I}\subset\mathcal{A}$, the sequence
\begin{align*}
\{0\}\longrightarrow C_\lambda^*(\Gamma)\otimes_{\mathrm{min}}\mathcal{I}\longrightarrow C_\lambda^*(\Gamma)\otimes_{\mathrm{min}}\mathcal{A}
\longrightarrow C_\lambda^*(\Gamma)\otimes_{\mathrm{min}}\left(\mathcal{A}/\mathcal{I}\right)\longrightarrow \{0\}
\end{align*}
is exact. Here $\otimes_{\mathrm{min}}$ stands for the minimal tensor product; see \cite[\S 4.1]{Pis} for details. This property turns out to be equivalent to Yu's Property A, which is probably better known; see \cite[Theorem 5.5.7]{BroOza}. The following theorem was essentially proved by Kraus \cite{Kra} and Haagerup--Kraus \cite{HaaKra}; we refer the reader to \cite[\S 12.4]{BroOza} for a detailed presentation.

\begin{thm}[Kraus, Haagerup--Kraus]\label{Thm_wa->ex}
Let $\Gamma$ be a discrete group. If $\Gamma$ is weakly amenable, then it is exact.
\end{thm}

This result was extended in \cite{Suz} to locally compact groups. We point out that exactness is a very weak property, and examples of non-exact groups are not easy to construct, but they do exist. The study of such groups was initiated in \cite{Grom}, and since then it has become a very active research area.

Another family of examples comes from algebraic groups. The following was proved in \cite[Theorem C]{LafdlS}.

\begin{thm}[Lafforgue--de la Salle]\label{Thm_Laf_dlS}
Let $F$ be a nonarchimedian local field and $n\geq 3$. Then $\operatorname{SL}(n,F)$ is not weakly amenable.
\end{thm}

In particular, this result applies to $\operatorname{SL}(n,\mathbb{Q}_p)$. Moreover, by Proposition \ref{Prop_CH_lattice}, the same holds for lattices in $\operatorname{SL}(n,F)$. This result was later extended to all $\tilde{A}_2$-lattices in \cite[Corollary E]{LedlSWi}; we refer the reader to \cite[Definition 9.1]{LedlSWi} and \cite[\S 1]{LedlSWi} for precise definitions.

\begin{thm}[L\'ecureux--de la Salle--Witzel]\label{Thm_LdlSW}
Let $\Gamma$ be an $\tilde{A}_2$-lattice. Then $\Gamma$ is not weakly amenable.
\end{thm}

\section{{\bf The Haagerup property and Cowling's conjecture}}\label{Sec_Haag_prop}

This section is devoted to an open problem regarding weak amenability that is often referred to as Cowling's conjecture. In order to state it precisely, we review the definition of the Haagerup property and how it generalises amenability.

\subsection{The Haagerup property}

Recall that a locally compact group $G$ is amenable if and only if its Fourier algebra has a bounded approximate identity; see Theorem \ref{Thm_Lep}. Furthermore, the approximate identity can be chosen to be positive definite and compactly supported; see \cite[Proposition 8.8]{Pie}. We record this fact in the following proposition.

\begin{prop}
A locally compact group $G$ is amenable if and only if there is a net of compactly supported, continuous, positive definite functions $(\varphi_i)$ on $G$ such that $\varphi_i$ converges to $1$ uniformly on compact subsets of $G$.
\end{prop}

As we already saw, one way of generalising this is by replacing the hypothesis of being positive definite by the condition
\begin{align*}
\sup_{i}\|\varphi_i\|_{M_0A(G)}<\infty,
\end{align*}
in which case we recover the definition of weak amenability. Another possibility is to relax the hypothesis of compact support.% We will denote by $C_0(G)$ the algebra of continuous function on $G$ that vanish at infinity.

We say that a locally compact group $G$ has the Haagerup property if there is a net of positive definite functions $(\varphi_i)$ in $C_0(G)$ such that $\varphi_i$ converges to $1$ uniformly on compact subsets of $G$. This property received that name because of Haagerup's seminal work \cite{Haa}, in which he showed that free groups satisfy it. We refer the reader to \cite{CCJJV} for a thorough account on the Haagerup property, its different characterisations and consequences.

\subsection{First version of Cowling's conjecture}
Both weak amenability and the Haagerup property generalise amenability, but they do so in quite different directions. Nonetheless, the examples of groups known to satisfy both properties, and the connections between completely contractive maps and completely positive maps suggest that there might be a link between them. More concretely, the following conjecture was proposed in \cite[\S 1.3.1]{CCJJV}.

\begin{conj}\label{Conj_Cow1}
Let $G$ be a locally compact group. Then $G$ has the Haagerup property if and only if $G$ is weakly amenable with $\boldsymbol\Lambda(G)=1$.
\end{conj}

This conjecture remained open for some time, until one of the directions was disproved. More precisely, the following was shown in \cite[Theorem 1.1]{CoStVa}.

\begin{thm}[Cornulier--Stalder--Valette]
Let $\Gamma$ and $\Lambda$ be countable, discrete groups. If both groups have the Haagerup property, then so does $\Lambda\wr\Gamma$.
\end{thm}

This theorem, together with Corollary \ref{Cor_Oza}, gives many examples of groups with the Haagerup property that are not weakly amenable; for example $\Z\wr\F_2$. Furthermore, this conjecture can fail even more dramatically. The following was proved in \cite[Theorem 2]{Osa}.

\begin{thm}[Osajda]\label{Thm_Osa}
There exist finitely generated groups acting properly on $\operatorname{CAT}(0)$ cube complexes that are not exact.
\end{thm}

Observe that, by Theorem \ref{Thm_GueHig} and Theorem \ref{Thm_wa->ex}, such cube complexes are necessarily infinite-dimensional. However, a group acting properly on such a complex will have the Haagerup property because the distance on it is a conditionally negative definite kernel; see \cite[Technical Lemma]{NibRee}.

\subsection{Second version of Cowling's conjecture}
One of the directions of Conjecture \ref{Conj_Cow1} is false, but the other one remains open. Hence we can restate it as follows.

\begin{conj}\label{Conj_Cow2}
Let $G$ be a locally compact group. If $G$ is weakly amenable with $\boldsymbol\Lambda(G)=1$, then it has the Haagerup property.
\end{conj}

All the examples of groups with $\boldsymbol\Lambda(G)=1$ presented in these notes are known to satisfy the Haagerup property, but no general proof of Conjecture \ref{Conj_Cow2} has been found. We now discuss some partial results or efforts towards proving this conjecture. Before that, we recall some characterisations of the Haagerup property; see \cite[\S 1.1.1]{CCJJV}.

\begin{thm}\label{Thm_char_Haagerup}
Let $G$ be a second countable, locally compact group. The following are equivalent:
\begin{itemize}
\item[(i)] The group $G$ has the Haagerup property.
\item[(ii)] There exists a continuous, conditionally negative definite function $\psi:G\to[0,\infty)$ such that
\begin{align*}
\lim_{x\to\infty}\psi(x)=\infty.
\end{align*}
\item[(iii)] There is a continuous, proper, isometric, affine action of $G$ on a Hilbert space.
\end{itemize}
\end{thm}

Recall that a kernel $\tilde{\psi}:X\times X\to[0,\infty)$ on a set $X$ is conditionally negative definite (cnd) if there is a Hilbert space $\mathcal{H}$ and a map $\xi:X\to\mathcal{H}$ such that
\begin{align*}
\tilde{\psi}(x,y)=\|\xi(x)-\xi(y)\|^2,\quad\forall x,y\in X.
\end{align*}
If $G$ is a group, we say that $\psi:G\to[0,\infty)$ is a cnd function if the kernel $\tilde{\psi}:G\times G\to[0,\infty)$ given by
\begin{align*}
\tilde{\psi}(x,y)=\psi(y^{-1}x),\quad\forall x,y\in G,
\end{align*}
is cnd; see \cite[\S II.C]{BedlHVa}. The following was proved in \cite[Theorem 1.2]{Knu2} and \cite[Proposition 4.3]{Knu2}.

\begin{thm}[Knudby]\label{Thm_Knudby}
Let $\Gamma$ be a countable, discrete group with $\boldsymbol\Lambda(\Gamma)=1$. Then there exist a function $\psi:\Gamma\to[0,\infty)$, a Hilbert space $\mathcal{H}$, and maps $R,S:\Gamma\to\mathcal{H}$ such that
\begin{align*}
\psi(y^{-1}x)=\|R(x)-R(y)\|^2+\|S(x)+S(y)\|^2,\quad\forall x,y\in \Gamma,
\end{align*}
and $\displaystyle\lim_{x\to\infty}\psi(x)=\infty$.
\end{thm}

In other words, we get a proper function that is the sum of a cnd kernel and a bounded function, but which is not necessarily cnd itself. This condition characterises the weak Haagerup property with constant $1$, as defined in \cite{Knu2}. The question then is whether one can replace this function by an actual cnd function.% In other words, is the Haagerup property equivalent to the weak Haagerup property with constant $1$?

Theorem \ref{Thm_Knudby} also allows one to construct affine actions in the spirit of Theorem \ref{Thm_char_Haagerup}.(iii), but the space needs to be changed. The following theorem is a consequence of more general results regarding embeddings of $L^p$ spaces and the properties of their norms; see \cite[Theorem 6.8]{CDH}, \cite[Corollary 6.23]{CDH} and the references therein. By an $L^p$ space, we mean a space of the form $L^p(\Omega,\mu)$ for some measure space $(\Omega,\mu)$.

\begin{thm}
Let $G$ be a second countable, locally compact group. Then $G$ has the Haagerup property if and only if it has a continuous, proper, isometric, affine action on a subspace of an $L^1$ space.
\end{thm}

The following was proved in \cite[Theorem 1.6]{Ver3} and \cite[Theorem 1.1]{Ver3}, using Theorem \ref{Thm_Knudby}.

\begin{thm}
Let $\Gamma$ be a countable, discrete group with $\boldsymbol\Lambda(\Gamma)=1$. Then, for every $\varepsilon>0$, $\Gamma$ has a proper, affine, $(1+\varepsilon)$-Lipschitz action on a subspace of an $L^1$ space.
\end{thm}

\section{{\bf The AP of Haagerup and Kraus}}\label{Sec_AP_HK}

In this section, we discuss a weaker property than weak amenability, introduced in \cite{HaaKra}. For this purpose, we need to look at $M_0A(G)$ as a dual Banach space. 

\subsection{Weakening weak amenability}
Let $G$ be a locally compact group. Observe that the duality $(L^1(G),L^\infty(G))$ provides an embedding of $L^1(G)$ into the dual space $M_0A(G)^*$. The following was proved in \cite[Proposition 1.10]{deCHaa}.

\begin{prop}[de Canni\`ere--Haagerup]
Let $G$ be a locally compact group, and let $Q(G)$ denote the completion of $L^1(G)$ in $M_0A(G)^*$. Then the dual space of $Q(G)$ can be identified with $M_0A(G)$ for the duality
\begin{align*}
\langle\varphi,f\rangle=\int_G\varphi(t)f(t)\,dt,\quad\forall\varphi\in M_0A(G),\ \forall f\in L^1(G).
\end{align*}
\end{prop}

This allows us to consider the weak* topology $\sigma(M_0A(G),Q(G))$ on $M_0A(G)$. Weak amenability can then be characterised by means of this topology; see \cite[Theorem 1.12]{HaaKra}.

\begin{thm}[Haagerup--Kraus]
Let $G$ be a locally compact group and $C\geq 1$. Then $G$ is weakly amenable with $\boldsymbol\Lambda(G)\leq C$ if and only if the constant function $1$ is in the $\sigma(M_0A(G),Q(G))$-closure of
\begin{align*}
\left\{\varphi\in A(G)\ \mid\ \|\varphi\|_{M_0A(G)}\leq C\right\}.
\end{align*}
\end{thm}

This motivates the following definition. We say that $G$ has the approximation property (AP) if the constant function $1$ is in the $\sigma(M_0A(G),Q(G))$-closure of $A(G)$ in $M_0A(G)$. Hence every weakly amenable group has the AP.

This property is not to be confused with the approximation property for Banach spaces described in Section \ref{Ssec_AP_Banach}. In some sense, the AP of Haagerup and Kraus is the operator algebraic version of the AP for Banach spaces. More precisely, the following was proved in \cite[Theorem 2.1]{HaaKra}.

\begin{thm}[Haagerup--Kraus]
Let $\Gamma$ be a discrete group. The following are equivalent:
\begin{itemize}
\item[(i)] The group $\Gamma$ has the AP.
\item[(ii)] The $\mathbf{C}^*$-algebra $C_\lambda^*(\Gamma)$ has the operator approximation property (OAP).
\item[(iii)] The von Neumann algebra $L\Gamma$ has the weak* operator approximation property (w*-OAP).
\end{itemize}
\end{thm}

These properties are defined as follows. We say that a $\mathbf{C}^*$-algebra $\mathcal{A}$ has the OAP if there is a net of finite rank maps $T_i:\mathcal{A}\to\mathcal{A}$ such that
\begin{align*}
\|(T_i\otimes\operatorname{id})x-x\|\to 0,\quad\forall x\in\mathcal{A}\otimes_{\mathrm{min}}\mathcal{K},
\end{align*}
where $\mathcal{K}$ denotes the algebra of compact operators on $\ell^2$. The w*-OAP for von Neumann algebras is defined in similar fashion; see \cite[\S 12.4]{BroOza} for details.

\subsection{Examples and non-examples}
Like weak amenability, the AP is stable under taking direct products and closed subgroups; see \cite[Proposition 1.14]{HaaKra}. One big difference between these properties can be seen in the following result; see \cite[Theorem 1.15]{HaaKra}.

\begin{thm}[Haagerup--Kraus]
Let $G$ be a locally compact group, and let $H$ be a closed, normal subgroup of $G$. If both $H$ and $G/H$ have the AP, then so does $G$.
\end{thm}

As a consequence, the semidirect product $\operatorname{SL}(2,\R)\ltimes\R^2$ satisfies AP, and we know from Theorem \ref{Thm_Oza_nonwa} that this group is not weakly amenable. Furthermore, for connected Lie groups, the question of which ones have AP is completely settled; see \cite[Theorem C]{HaKndL}.

\begin{thm}\label{Thm_AP_Lie}
Let $G$ be a connected Lie group, and let $G=RS$ be a Levi decomposition, where $S$ is locally isomorphic to a direct product $S_1\times\cdots\times S_n$ of connected simple factors. Then $G$ has the AP if and only if, for every $i\in\{1,\ldots,n\}$, the real rank of $S_i$ is at most $1$.
\end{thm}

Again, this theorem is a compilation of several works, that we now briefly describe. First, since weak amenability implies AP, some cases were already settled in the works cited in Section \ref{Sec_Lie}. The main question was to prove the failure of AP in higher rank. The first breakthrough was achieved in \cite{LafdlS}, where it was shown that $\operatorname{SL}(3,\R)$ does not satisfy AP. This was extended in \cite{HaadLa} to all higher rank simple Lie groups with finite centre; the finite centre condition was subsequently removed in \cite{HaadLa2}. Finally, the complete characterisation in Theorem \ref{Thm_AP_Lie} was established in \cite{HaKndL}; see also \cite{Knu4} for a different approach.

As was mentioned in Section \ref{Sec_obstr}, some of the results presented there were actually proved for the AP. More concretely, in Theorems \ref{Thm_wa->ex}, \ref{Thm_Laf_dlS} and \ref{Thm_LdlSW}, one can replace weak amenability by AP, and the results are still valid.

\subsection{The convoluters and pseudo-measures problem}
Let $G$ be a locally compact group endowed with a left Haar measure, and let $\rho:G\to\mathbf{U}(L^2(G))$ be the right regular representation:
\begin{align*}
\rho(s)f(t)=f(ts)\Delta(s)^{1/2},\quad\forall s,t\in G,\ \forall f\in L^2(G),
\end{align*}
where $\Delta:G\to(0,\infty)$ denotes the modular function. As a consequence of the Kaplansky density theorem (see e.g. \cite[Theorem 4.3.3]{Mur}), the group von Neumann algebra $LG$ can be identified with the commutant of $\rho(G)$. More precisely,
\begin{align*}
LG=\rho(G)'=\{T\in\mathbf{B}(L^2(G))\ \mid\ T\rho(s)=\rho(s)T,\ \forall s\in G\}.
\end{align*}
The convoluters and pseudo-measures problem asks whether this fact still holds in the context of operators on $L^p$ spaces. Let $p\in(1,\infty)$. We denote by $\lambda_p$ and $\rho_p$ the left and right regular representations on $L^p(G)$. In other words,
\begin{align*}
\lambda_p(s)f(t)&=f(s^{-1}t), & \rho_p(s)f(t)&=f(ts)\Delta(s)^{1/p},
\end{align*}
for all $s,t\in G$ and $f\in L^p(G)$. We define the algebra of $p$-pseudo measures $PM_p(G)$ as the weak*-closed linear span of $\lambda_p(G)$ in $\mathcal{B}(L^p(G))$. Here $\mathcal{B}(L^p(G))$ is viewed as the dual space of the projective tensor product $L^{p'}(G)\hat{\otimes}L^p(G)$, where $p'$ denotes the H\"older conjugate of $p$; see \cite[\S 2.2]{Rya} for details. The algebra of $p$-convoluters $CV_p(G)$ is simply defined as the commutant of $\rho_p(G)$. Since the representations $\lambda_p$ and $\rho_p$ commute with each other, we always have
\begin{align*}
PM_p(G)\subseteq CV_p(G).
\end{align*}
The following was proved by Daws and Spronk in \cite[Theorem 1.1]{DawSpr}.

\begin{thm}\label{Thm_conv_pseu}
Let $G$ be a locally compact group satisfying AP. Then, for every $p\in(1,\infty)$,
\begin{align*}
PM_p(G)=CV_p(G).
\end{align*}
\end{thm}

Once again, this result has a long history. A proof for amenable groups was presented in \cite[\S 4]{Eym2}, where it is attributed to Herz. It was subsequently extended by Cowling \cite{Cow2} to some rank 1 Lie groups, and then to all weakly amenable groups in \cite{Cow3}. In that same paper, Theorem \ref{Thm_conv_pseu} was announced with a hint of proof. A complete proof was presented in \cite{DawSpr}.

It is still unknown whether there is a group and a value of $p$ for which the inclusion $PM_p(G)\subset CV_p(G)$ is strict.

\bibliographystyle{plain} 

\bibliography{Bibliography}

\end{document}